\documentclass[12pt]{elsarticle}
\usepackage{lineno,hyperref}
\usepackage{amsmath,amssymb,amsthm}
\usepackage{txfonts}
\usepackage{amssymb}
\usepackage[latin1]{inputenc}
\usepackage{version,tabularx,multicol}
\usepackage{graphicx,float,psfrag}
\usepackage{stmaryrd}
 \usepackage{etoolbox,xstring,mfirstuc,textcase}
 \theoremstyle{plain}
\newtheorem{theorem}{Theorem}[section]

 \newtheorem{corollary}[theorem]{Corollary}

 \newtheorem{proposition}[theorem]{Proposition}

 \newtheorem{lemma}[theorem]{Lemma}

\newtheorem{remark}{Remark}[section]

\allowdisplaybreaks[4]

\newcommand{\ind}{{\bf 1}}
  
  \makeatletter
  \@addtoreset{equation}{section}
  \makeatother

 \def\beqlb{\begin{eqnarray}}\def\eeqlb{\end{eqnarray}}
 \def\beqnn{\begin{eqnarray*}}\def\eeqnn{\end{eqnarray*}}

\newcommand{\bcen}{\begin{center}}
\newcommand{\ecen}{\end{center}}
\newcommand{\bgeqn}{\begin{equation}}
\newcommand{\edeqn}{\end{equation}}

\modulolinenumbers[5]

\journal{Journal of \LaTeX\ Templates}

\begin{document}

\begin{frontmatter}

\title{On empty balls of critical 2-dimensional branching random walks}

\author[address]{Shuxiong Zhang}
\ead{shuxiong.zhang@mail.bnu.edu.cn}
\address[address]{School of Mathematics and Statistics, Anhui Normal University}


\begin{abstract}

Let $\{Z_n\}_{n\geq 0 }$ be a critical $d$-dimensional branching random walk started from a Poisson random measure whose intensity measure is the Lebesgue measure on $\mathbb{R}^d$. Denote by $R_n:=\sup\{u>0:Z_n(\{x\in\mathbb{R}^d:|x|<u\})=0\}$ the radius of the largest empty ball centered at the origin of $Z_n$. In \cite{reves02}, R\'ev\'esz (2002) conjectured that if $d=2$, then

$$\lim_{n\to\infty}\frac{R_n}{\sqrt n}\overset{\text{Law}}=\text{non-trival~distribution}.$$
This work confirms the conjecture. It turns out that the limit distribution can be precisely characterized through the super-Brownian motion. Moreover, we also give complete results of empty balls of the branching random walk with infinite second moment offspring law. As a by-product, this article also improves the assumption of maximal displacements of branching random walks in Lalley and Shao \cite[Theorem 1]{lalley2015}.
\end{abstract}


\end{frontmatter}

\section{Introduction and Main results}
\subsection{Introduction}
 In this paper, we consider a branching random walk $\{Z_n\}_{n\geq0}$ started from the Poisson random measure with Lebesgue intensity. This model is a measure-valued Markov process, which is governed by a probability distribution $\{p_k\}_{k\geq 0}$ on natural numbers (called the offspring distribution) and an $\mathbb{R}^d$-valued random vector $X$ (called the step size or displacement). Let's define it as follows.
At generation $0$, there exist infinitely many particles distributed according to the Poisson random measure with Lebesgue intensity. At generation $1$, each particle first produce children independently according to $\{p_k\}_{k\geq 0}$  and then every children  performs a jump independently of each others from the place where they are born according to the law of the step size $X$, where the branching and motion are independent. The point process $Z_1$ is formed by all the
particles and their positions alive at generation $1$.  At generation $2$, the particles alive at generation $1$ repeat their parent's behaviour independently from the place where they are born and the procedure goes on.
	\par
	
	Since there exist countably many particles at time $0$, we can let $\mathbb{T}^i$, $i\geq1$ be the genealogical tree of the branching process
	rooted by the $i$-th particle of $Z_0$; see Shi \cite[p. 13-14]{shizhan2015}.  Let $\mathbb{T}=\cup^{\infty}_{i=1}\mathbb{T}^i$ and
	$X_{u}$ be the displacement of particle $u$. For each particle $u\in\mathbb{T}$, we use $|u|$ to denote the generation of $u$. If $|u|=0$, then $X_u$ is determined by the Poisson random measure $Z_0$. Otherwise, $\{X_u: u\in\mathbb{T}, |u|\neq 0\}$ are i.i.d. copies of the step size $X$.
	 Let $S_{\omega}:=\sum_{u\preceq\omega}X_u$ be the position of particle $\omega\in\mathbb{T}$, where $u\preceq \omega$
	means that $u=\omega$ or $u$ is an ancestor of $\omega$.  The
	population and their locations
	 at generation $n$ forms a point process given by
	$$Z_n=\sum_{|\omega|=n}\delta_{S_{\omega}}.$$
\par
Let $m:=\sum_{k\geq0}kp_k$ be the mean of the offspring law. We call $\{Z_n\}_{n\geq0}$
 a supercritical (critical, subcritical) branching random walk if $m>1$ ($= 1$, $< 1$). In this work, we always assume that
 $$m=1, p_1<1~\text{and}~\mathbb{E}[X]=0.$$
Moreover, we always assume that the step size $X$ has a positive definite covariance matrix
$${\bf C}=(c_{i,j})_{1\leq i,j\leq d }.$$
Let $B(u):=\{x\in\mathbb{R}^d:|x|<u\}$ be the $d$-dimensional open ball with radius $u$, where $|x|$ stands for the Euclidean norm of the vector $x$. Denote by $$R_n:=\sup\{u>0:Z_n(B(u))=0\}$$ the radius of the largest empty ball centered at the origin of $Z_n$, where we take $0=\sup\emptyset$ by convention. Or, in other words, $R_n$ is the shortest distance of the particles at time $n$ from the origin.

\par
The research on the empty ball was first studied by R\'ev\'esz \cite{reves02} for the binary branching Wiener process (i.e. $p_0=p_2=1/2$ and $X$ is a standard normal random vector) started from $\text{PRM}(\lambda)$, where we use $\text{PRM}(\lambda)$ to stand for the Poisson random measure with the $d$-dimensional Lebesgue measure $\lambda$ as its intensity measure. R\'ev\'esz \cite{reves02} proved that if $d=1$, then $R_n/n$ converges in law to an exponential random variable as $n\to\infty$. For high-dimensional cases, he presented following two conjectures (see \cite[Conjecture 1]{reves02} ):\\
\noindent (i) If $d=2$, then for any $r\in(0,\infty)$,
\begin{align}\label{5ju8iw23}
\lim_{n\to\infty}\mathbb{P}\left(\frac{R_n}{\sqrt n}\geq r\right)=e^{-F_2(r)}\in(0,1),
\end{align}
where $F_2(r)$ satisfies
$$\lim_{r\to\infty} \frac{F_2(r)}{\pi r^2}=1;$$
(ii) If $d\geq 3$, then for any $r\in(0,\infty)$,
\begin{align}\label{4redff}
\lim_{n\to\infty}\mathbb{P}(R_n\geq r)=e^{-F_d(r)}\in(0,1),
\end{align}
where $F_d(r)$ satisfies
\begin{align}\label{4redff12}
\lim_{n\to\infty}\frac{F_d(r)}{C_d r^{d-2}}=1,
\end{align}
and $\lim_{d\to\infty}C_d/\big[\frac{\pi^{d/2}}{\Gamma(d/2+1)}\big]^{(d-2)/d}=1$ ($\Gamma(\cdot)$ is the Gamma function). Later, Hu \cite[Theorem 1.2]{hu05} partially confirmed R\'ev\'esz's conjecture for $d\geq3$ by showing that $\lim_{n\to\infty}\mathbb{P}(R_n\geq r)$ exists in $(0,1)$. But, (\ref{4redff12}) remains unproven.

 \par
Recently, Xiong and Zhang \cite{xz21} studied empty balls for critical super-Brownian motions in any dimension. For the super-Brownian motion, they showed that in $d\geq 2$, the limit distributions are consistent with R\'ev\'esz's conjectures. Later, by allowing the branching mechanism of the super-Brownian motion is critical, subcritical or has infinite variance, Liu, Xiong and Zhang \cite{zhang25jtp} generalized the work of \cite{xz21} to super-Brownian motions with general branching mechanisms. Theses results of super-Brownian motions \cite{xz21}\cite{zhang25jtp} are complete and refined. Nevertheless, to the best of our knowledge, for the branching Wiener process (or the branching random walk) model, R\'ev\'esz's conjecture for the $2$-dimensional case remains open until now. This article confirms and generalizes this conjecture.
\par

Xiong and Zhang \cite{xzacta} also generalized the results of the branching Wiener process \cite{hu05}\cite{reves02} to the branching random walk by assuming that all components of the step size are independent and centered. If the offspring law is critical and has finite variance, then the results of \cite[Theorems 1.1, 1.2, 1.5]{xzacta} are as follows.

\medskip

\noindent (i)Assume $m=1$, $d=1$, $\sigma^2:=\sum_{k\geq0}k^2p_k-(\sum_{k\geq0}kp_k)^2<\infty$ and $\mathbb{E}[|X|^{\alpha}]<\infty$ for some $\alpha>2$. Then, for $r>0$,

$$\lim_{n\to\infty}\mathbb{P}\left(\frac{R_n}{n}\geq r\right)=e^{-\frac{4r}{\sigma^2}}.$$

\noindent(ii)Assume $m=1$, $d=2$, $\sigma^2<\infty$ and $\mathbb{E}[|X|^4]<\infty$. Then, for any $r>0$,
\begin{align}\label{4hyhy7ht53r}
0<\liminf_{n\to\infty}\mathbb{P}\left(\frac{R_n}{\sqrt n}\geq r\right)\leq\limsup_{n\to\infty}\mathbb{P}\left(\frac{R_n}{\sqrt n}\geq r\right)<1.
\end{align}

\noindent (iii)Assume $m=1$, $d\geq3$, $\sigma^2<\infty$ and $\mathbb{E}[|X|^3]<\infty$. Then, there exists a positive function $F_d(r)$, $r>0$ such that

$$\lim_{n\to\infty}\mathbb{P}\left(R_n\geq r\right)=e^{-F_d(r)}.$$
Moreover,
$$\frac{v_d(r)}{1+\sigma^2 C_d(r)r^2}\leq F_d(r)\leq v_d(r),$$
where $C_d(r)\in(0,\infty)$ is a constant and $v_d(r)$ is the volume of a $d$-dimensional ball with radius $r$.

\medskip

\noindent From above results, one can see that the assumption $\sigma^2<\infty$ is necessary. If not the case, things will become different. Let $L(s),~s\in[0,1]$  be a slowly varying function at $0$. \cite[Theorem 1.8]{xzacta} is stated as follows. Assume $m=1$, $\sum^{\infty}_{k=0}p_ks^k=s+(1-s)^{1+\beta}L(1-s)$, $s\in[0,1]$ for some $\beta\in\left(0,\frac{1}{d}\right]$ and $\mathbb{E}[|X|^{\alpha}]<\infty$ for some $\alpha>(\beta+1)d$. Then, for $r>0$,
\begin{align}\label{78igt53r}
\lim_{n\to\infty}\mathbb{P}\left(\frac{R_n}{b_n}\geq r\right)=\exp\left\{-v_d(r)(\beta^{-1})^{\beta^{-1}}\right\},
\end{align}
where
$$b_n:=\Big[nL\big(\mathbb{P}_{\delta_0}(Z_n(\mathbb{R}^d)>0)\big)\Big]^{\frac{1}{\beta d}},~n\geq1.$$
The above results consider the critical case $m=1$. For subcritical case, \cite[Theorem 1.9]{xzacta} showed that $R_n$ increases exponentially, which is stated as follows. Assume $m\in(0,1)$, $d\geq1$, $\sum_{k\geq0}p_k(k\log k)<\infty$ and $\mathbb{E}\left[|X|^{\alpha}\right]<\infty$ for some $\alpha>1$. Then, for $r>0$,
$$\lim_{n\to\infty}\mathbb{P}\left(\frac{R_n}{(1/m)^\frac{n}{d}}\geq r\right)=e^{-Q(0)v_d(r)},$$
where $Q(0)\in(0,\infty)$ is a constant.

\par
In this article, we shall refine the work of Xiong and Zhang \cite{xzacta} in three aspects. Firstly, we relax the assumption of the independence between each component of the step size. Secondly, by using the maximal displacement of branching random walks, we strength (\ref{4hyhy7ht53r}) by showing the existence of the limit. As a by-product, this work also improves the assumption of Lalley and Shao [8, Theorem 1]. Thirdly, in fact, for any $\beta\in(0,1)$, the offspring law satisfying $\sum^{\infty}_{k=0}p_ks^k=s+(1-s)^{1+\beta}L(1-s)$ implies $\sigma^2=\infty$. Since (\ref{78igt53r}) assumes $\beta\in\left(0,\frac{1}{d}\right]$, the result is not complete.  By using large deviation probabilities of random walks and the spine decomposition method, we show that if $\sigma^2=\infty$, then the empty ball has three kinds of phase transition phenomenons, which completes the results of \cite[Theorem 1.8]{xzacta}. Meanwhile, we use the size-biased branching random walk to characterize the limit distribution of empty balls of high dimensional ($d\geq3$) critical branching random walks, thereby improving upon the result in Hu \cite[Theorem 1.2]{hu05}, which does not provide an explicit expression for the limit.

\subsection{Main results}

In this section, we will give our main results. Since the statement of theorems involve the super-Brownian motion, we first give a brief introduction to it as follows.
\par
Let $\psi$ be a function of the form
   $$\psi(u)=au+bu^2+\int^{\infty}_{0}(e^{-ru}-1+ru)n(dr), u\geq0,$$
where $a\in\mathbb{R},b\geq0$ and $n$ is a $\sigma$-finite measure on $(0,\infty)$ such that $\int_{(0,\infty)}r\wedge r^2 n(dr)<\infty$. The super-Brownian motion $\{X_t\}_{t\geq0}$ with initial measure $\mu$ and branching mechanism $\psi$ is a measure-valued Markov process, whose transition probabilities are characterized through their Laplace transforms. For any bounded non-negative function $\phi$ and $t\geq0$, we have
   \begin{align}\label{ujytas}
  \mathbb{E}_\mu\left[e^{-\int_{\mathbb{R}^d}\phi(x)X_t(dx)}\right]=e^{-\int_{\mathbb{R}^d} u(t,x)\mu(dx)},
   \end{align}
  where $u(t,x)$ is the unique positive solution to the following nonlinear partial differential equation:
  \begin{align}\label{6hjujufr4}
   \begin{cases}
   \frac{\partial u(t,x)}{\partial t}=\mathcal{L}u(t,x)-\psi(u(t,x)),\cr
   u(0,x)=\phi(x).
   \end{cases}
   \end{align}
In above, $\mathcal{L}:=\sum_{1\leq i,j\leq d }\frac{c_{i,j}}{2}\frac{\partial^2 }{\partial x_i \partial x_j}$ is the
generator of a $d$-dimensional Brownian motion, whose covariance matrix at time $1$ is ${\bf C}=(c_{i,j})_{1\leq i,j\leq d }$. We refer the reader to \cite{Etheridge}, \cite{Li} and \cite{perkins} for a more detailed overview to super-Brownian motions. In this paper, for a measure $\mu$, we always use $\mathbb{E}_{\mu}$ to denote the expectation with respect to $\mathbb{P}_{\mu}$, the probability measure under which the branching random walk $\{Z_n\}_{n\geq0}$ (or the super-Brownian motion $\{X_t\}_{t\geq0}$) has initial value $Z_0=\mu$ (or $X_0=\mu$). For ease of notation, we write $\mathbb{P}=\mathbb{P}_{\text{PRM}(\lambda)}.$

\par
Recall that $\sigma^2=\sum_{k\geq0}k^2p_k-(\sum_{k\geq0}kp_k)^2$ is the variance of the offspring law. For $x\in\mathbb{R}^d$ and $r>0$, let $B(x,r):=\{z\in\mathbb{R}^d:|z-x|<r\}$ and $x^{(i)}$ be the $i$-th component of $x$, $1\leq i\leq d$. Set $y^+:=\max\{y,0\}$ for $y\in\mathbb{R}.$ Recall that we always consider the critical branching random walk (i.e. $\sum_{k\geq0}kp_k=1$) and the step size is centered.
\begin{theorem}\label{thdim2}Assume $d=2$ and $\sigma^2<\infty$. If $\mathbb{E}[((X^{(1)})^+)^4+((X^{(2)})^+)^4]<\infty$, then for any $r>0$,
\begin{align}
 \lim_{n\to\infty}\mathbb{P}\left(\frac{R_n}{\sqrt n}\geq r\right)= \exp\left\{\int_{\mathbb{R}^2}\log\mathbb{P}_{\delta_0}\left(X_1(B(x,r))=0\right)dx\right\}\in(0,1),\nonumber
\end{align}
where $\{X_t\}_{t\geq0}$ is a two-dimensional super-Brownian motion with branching mechanism $\psi(u)=\sigma^2u^2$.
\end{theorem}
\begin{remark}\label{4reer43}
Though \cite[Theorem 1.2]{xz21} only deal with the case that the generator $\mathcal{L}$ defined in (\ref{6hjujufr4}) is a Laplace operator, one can easily imitate its proof to show that the theorem still holds for general $\mathcal{L}$. Because the proof of \cite[Theorem 1.2]{xz21} mainly use the scale invariance property of $\{X_t\}_{t\geq0}$, which can be easily proved for general $\mathcal{L}$ by imitating proofs of \cite[p. 52]{Etheridge}. From \cite[Theorem 1.2]{xz21}, we have
$$\lim_{r\to\infty}\frac{\int_{\mathbb{R}^2}-\log\mathbb{P}_{\delta_0}\left(X_1(B(x,r))=0\right)dx}{\pi r^2/\sigma^2}=1.$$
If the offspring law of $\{Z_n\}_{n\geq0}$ satisfies $p_0=p_2=\frac{1}{2},$ then $\sigma^2=1$. Thus, comparing with (\ref{5ju8iw23}), Theorem \ref{thdim2} confirms \cite[Conjecture 1]{reves02} of $d=2$ in a very general setting.
\end{remark}

\par
The following theorem considers the case of $\sigma^2=\infty$. To state the theorem, we need introduce the following notations. Let $L(s)$, $s\in[0,1]$ be a slowly varying function as $s\to0+$ (i.e. $\lim_{s\to0+}L(sx)/L(s)=1$ for any $x\in(0,1]$). Recall that $v_d(r)=\frac{\pi^{d/2}r^d}{\Gamma(d/2+1)}$ (where $\Gamma(\cdot)$ is the so-called Gamma function) is the volume of a
$d$-dimensional ball with radius $r$. Let $Q^*$ be the probability under which the size-biased critical branching random walk $\{Z_n\}_{n\geq0}$ with spine $\{w_n\}_{n\geq0}$ and $Z_0=\delta_0$ is defined. Recall that $S_u$ is the position of particle $u$ and $X_u$ is the displacement of $u$. Let $b(w_k)$ be the set of brothers of the spine particle $w_{k}$. Let $\{Z^u_k\}_{k\geq0}$ be the sub-branching random walk emanating from particle $u$ with $Z^{u}_0:=\delta_0$. We refer the reader to Case 3 of Section \ref{8i87ujr4g} below for a more detailed explanation of above notations.
\begin{theorem}\label{thdim4} Assume $d\geq1$.  \\
(i) If $\sum^{\infty}_{k=0}s^kp_k=s+(1-s)^{1+\beta}L(1-s)$ with $\beta\in(0,\frac{2}{d}\wedge1)$ and $\mathbb{E}[|X|^{\alpha}]<\infty$ for some $\alpha>(\beta d+d)\vee 2$, then
\begin{align}
\lim_{n\to\infty}\mathbb{P}\left(\frac{R_n}{b_n}\geq r\right)=\exp\left\{-v_d(r)(\beta^{-1})^{\beta^{-1}}\right\},\nonumber
\end{align}
where
\begin{align}
b_n:=[nL(\mathbb{P}_{\delta_0}(Z_n(\mathbb{R}^d)>0))]^{\frac{1}{\beta d}},~n\geq1.\nonumber
\end{align}
(ii) If $\lim_{n\to\infty} n^{^{1+\beta}}\sum^{\infty}_{k=n}p_k=:\kappa(\beta)\in(0,\infty)$ with $\beta=\frac{2}{d}\in(0,1)$ and $\sum^d_{i=1}\mathbb{E}[((X^{(i)})^+)^{2+\frac{2}{\beta}}]<\infty$, then
\begin{align}
\lim_{n\to\infty}\mathbb{P}\left(\frac{R_n}{\sqrt n}\geq r\right)=\exp\left\{\int_{\mathbb{R}^d}\log\mathbb{P}_{\delta_0}\left(X_1(B(x,r))=0\right)dx\right\}\in(0,1),\nonumber
\end{align}
where $\{X_t\}_{t\geq0}$ is a $d$-dimensional super-Brownian motion with the stable branching mechanism $\psi(u)=\frac{1}{\beta}{\kappa(\beta)\Gamma(1-\beta)}u^{1+\beta}$.\\
(iii) If $\sum^{\infty}_{k=0}k^{1+\beta}p_k<\infty$ with $\beta\in(\frac{2}{d},1]$ and $\mathbb{E}[|X|^2]<\infty$, then
\begin{align}
\lim_{n\to\infty}\mathbb{P}({R_n}\geq r)=e^{-I_{B(r)}}\in(0,1),\nonumber
\end{align}
where
$$
I_{B(r)}:=\mathbb{E}_{Q^*}\left[\int_{B(r)}\Big[{1+\sum^{\infty}_{k=1}\sum_{u\in b(w_{k+1})}Z^u_{k-1}(B(r)-y+S_{w_k}-X_u)}\Big]^{-1}dy\right].
$$
\end{theorem}
\begin{remark}\label{46gtyy67m}
Theorem \ref{thdim4} (i) improves the assumption $\beta\in(0,\frac{1}{d})$ in \cite[Theorem 1.8]{xzacta} to the optimal assumption $\beta\in(0,\frac{2}{d}\wedge1).$ In \cite[Theorem 1.5]{xzacta}, we assume the step size satisfies $\mathbb{E}[|X|^3]<\infty$. Theorem \ref{thdim4} (iii) weakens the assumption to the finite second moment condition.
\end{remark}
\begin{remark}
Rapenne \cite[Theorem 1.3]{rapenne2023} proved that if $X$ is a $d$-dimensional standard normal random vector or $X$ is lattice, bounded, symmetric and aperiodic, then
$Z_n$ converges weakly to some non-degenerate random variable $\Lambda_{\infty}$ in the sense of vague topology as $n\to\infty$, where the term  $I_{B(r)}$ in Theorem \ref{thdim4} (iii) is used to describe to the characteristic functional of $\Lambda_{\infty}$. In other words, for any continuous function $f$ with a compact support on $\mathbb{R}^d$ and $\lambda\in \mathbb{R}^d$, it follows that
\begin{align}
\lim_{n\to\infty}\mathbb{E}[e^{i\lambda Z_n(f)}]=\lim_{n\to\infty}\mathbb{E}[e^{i\lambda \Lambda_{\infty}(f)}].
\end{align}
This means $Z_n(f)$ converges weakly to $\Lambda_{\infty}(f)$ as $n\to\infty$. Let $f$ be a non-negative continuous function with a compact support such that $\{x\in \mathbb{R}^d:f(x)>0\}=B(r)$. Then, by Theorem \ref{thdim4} (iii), it entails that
\begin{align}
\lim_{n\to\infty}\mathbb{P}({R_n}\geq r)=\lim_{n\to\infty}\mathbb{P}(Z_n(B(r))=0)=\lim_{n\to\infty}\mathbb{P}(Z_n(f)=0)=e^{-I_{B(r)}}.
\end{align}
Since $\ind_{\{0\}}(y), y\in\mathbb{R}$ is not a suitable test function for the weak convergence, Theorem \ref{thdim4} (iii) is not a direct consequence of \cite[Theorem 1.3]{rapenne2023}. Moreover, using the method of Theorem \ref{thdim4} (iii), one can weaken assumptions of the step size in \cite[Theorem 1.3]{rapenne2023}.
\end{remark}

 The main difficulty of Theorem \ref{thdim2} lies in the lower bound. The precise upper bound can be easily obtained by using the fact that the super-Brownian motion is a scaling limit of the branching random walk; see \cite[(3.12)]{xzacta}. For the lower bound, since the proofs involve the indicative function, which is not continuous, the desired lower bound can not be directly obtained through the scaling limit; see \cite[Remark 1.4]{xzacta}. To overcome this difficulty, we prove the lower bound with three steps. Firstly, we study the maximal displacement of the critical branching random walk, which improves the assumption of \cite[Thorem 1]{lalley2015}. Our method to the proof mainly use maximums of random walks and a stopping time trick, which is different from the overshoot distribution strategy used in \cite{lalley2015}. Secondly, by using above results for maximal displacements, we obtain an uniform upper bound of the tail probability of the maximum of a 1-dimensional critical branching random walk at some finite time; see Lemma \ref{566u9g13q} below. Finally, we use the maximum of the 1-dimensional critical branching random walk to dominate the local hitting probability of the 2-dimensional branching random walk (see (\ref{465gth51}) below), which is crucial in arguing that the empty ball probability involved can be bounded above by a local extinction probability of the super-Brownian motion. By above arguments, we have completely confirmed and generalized this conjecture.
  \par

For our purposes of proving Theorem \ref{thdim2}, the following theorem gives the asymptotic behaviour of the maximal displacement, which improves the assumption of \cite[Theorem 1]{lalley2015}. Write $u \in Z_i$ if $u$ is a
particle at time $i$ in the branching random walk $\{Z_n\}_{n\geq0}$. Recall that $S_u$ is the position of particle $u$. Denote by $M:=\max_{n\geq0}\max_{u\in Z_n}S_u$ the maximal displacement of the branching random walk. Put $\eta^2=\mathbb{E}[X^2]$.

\begin{theorem} \label{thdimmaxplace}Assume $d=1$, $\sigma^2<\infty$, $\eta^2<\infty$ and $\mathbb{E}[(X^{+})^4]<\infty$. Then,
\begin{align}\label{45ghyh89j}
\lim_{x\to+\infty}x^2\mathbb{P}_{\delta_0}(M\geq x)=\frac{6\eta^2}{\sigma^2}.
\end{align}
\end{theorem}

\begin{remark}Lalley and Shao \cite[Theorem 1]{lalley2015} proved that if $\sum^{\infty}_{k=0}k^3p_k<\infty$, $\mathbb{E}[|X|^{4+\varepsilon}]<\infty$ for some $\varepsilon>0$ and $X$ is integer-valued, then (\ref{45ghyh89j}) holds. Moreover, they \cite[p. 76]{lalley2015} commented that $4-\varepsilon$ moments of $X$ are not enough for (\ref{45ghyh89j}).
\end{remark}
\begin{remark}
For the case of $\sigma^2=\infty$, Hou et al. \cite{houbernoulli} show that for a critical branching L\'evy process with $(1+\beta)$-stable branching mechanism (where $\beta\in(0,1)$), if $\mathbb{E}[(X^+)^r]<\infty$ for some $r>\frac{2(1+\beta)}{\beta}$ , then $x^{2/\beta}\mathbb{P}_{\delta_0}(M\geq x)$ converges to some positive constant as $x\to\infty$. If $\mathbb{E}[X^2]=\infty$, Lalley and Shao \cite{lalley2016} and Profeta \cite{Profeta22} proved that for a critical branching $\alpha$-stable process (where $\alpha\in(0,2)$) if $\sum^{\infty}_{k=0}k^3p_k<\infty$, then $x^{{\alpha/2}}\mathbb{P}_{\delta_0}(M\geq x)$ converges to some positive constant as $x\to\infty$.
\end{remark}
The rest of this paper is organised as follows. We first study the maximum of the 1-dimensional branching random walk in Section \ref{sectionthdim2}, where Theorem \ref{thdimmaxplace} is proved. Then, in Section \ref{45gttghy9hn3r}, we use the maximum of the branching random walk to prove Theorem \ref{thdim2}. Finally, the empty ball of the branching random walk whose offspring has infinite second moment is studied in Section \ref{8i87ujr4g}, where Theorem \ref{thdim4} is proved.
\section{Proof of Theorem \ref{thdimmaxplace}: Maximal Displacement}\label{sectionthdim2}

In this section, we are going to prove Theorem \ref{thdimmaxplace}. To this end, we should give some preliminary results.
In the sequel of this article, write $\mathbf{P}:=\mathbb{P}_{
\delta_0}$, $\mathbf{E}:=\mathbb{E}_{\delta_0}$ and $|Z_n|:=Z_n(\mathbb{R}^d)$ for short. We assume $d=1$ in the remainder of this section. Put
$$u(x):=\mathbf{P}(M\geq x)~\text{for}~x\in\mathbb{R}.$$
 For particle $\omega,v$, write $v\prec \omega$, if $v$ is an ancestor of $\omega$, and $v\preceq \omega$ if $v=\omega$ or $v\prec \omega$. Recall that $X_{\omega}$ is the displacement of particle $\omega$. Let $M_{\omega}:=\max_{v:\omega\preceq v}(S_v-S_{\omega})$ be the maximal displacement of the sub-branching random walk emanating from $\omega$. Note that $M_{\omega}$ and $M$ has the same distribution. Let $M^*$ be an i.i.d. copy of $M$ and independent of $X$. Then, by the Markov property, for $x>0$,
\begin{align}
1-u(x)&=\mathbf{P}\left(M<x\right)\cr
&=\mathbf{E}\left[\mathbf{P}(M< x, |Z_1|=0\Big{|}\sigma(|Z_1|))\right]+\mathbf{E}\left[\prod_{v\in Z_1}\mathbf{P}(M_v+X_v< x,X_v<x, |Z_1|>0\Big{|}\sigma(|Z_1|))\right]\cr
&=p_0\mathbf{P}\left(0<x\right)+\sum^{\infty}_{n=1}p_n\mathbf{E}\left[\mathbf{P}\left(M^*+X<x,X<x|X\right)^n\right]\cr
&=p_0+\sum^{\infty}_{n=1}p_n\mathbf{E}\left[(1-u(x-X))^n\ind_{\{X<x\}}\right]\cr
&=p_0+\sum^{\infty}_{n=1}p_n\mathbf{E}\left[(1-u(x-X))^n\right]\cr
&=\sum^{\infty}_{n=0}p_n\mathbf{E}\left[\big(1-u(x-X)\big)^n\right]\cr
&=\mathbf{E}[(1-u(x-X))^{|Z_1|}],\nonumber
\end{align}
where the $5$-th equality comes from $u(x)=1$ for $x\leq0$, and the $6$-th equality follows from the convention $0^0=1$. Let $Q(s):=1-\mathbf{E}[(1-s)^{|Z_1|}]$, $s\in[0,1]$. Above yields that
\begin{align}\label{45kopq1}
u(x)=Q(\mathbf{E}[u(x-X)]),~x>0.
\end{align}
Set
$$f(s):=\mathbf{E}[s^{|Z_1|}]=\sum_{k\geq0}p_ks^{k},~H(s):=[s-1+f(1-s)]/s, s\in[0,1],$$
where we define $f(0):=\lim_{s\to0+}f(s)=p_0$ and $H(0):=\lim_{s\to0+}H(s)=0.$

 The following lemma comes from \cite[Proposition 6]{lalley2015}. We note that the lattice assumption of \cite[Proposition 6]{lalley2015} can be removed. Let $\{W_n\}_{n\geq0}$ be a random walk with increment following the distribution of $-X$. In this paper, we always use $\mathbb{E}_{x}$ to denote the expectation with respect to $\mathbb{P}_{x}$, the probability measure under which the random walk $\{W_n\}_{\geq 0}$ (or the standard Brownian motion $\{B_t\}_{t\geq0}$) has initial value $W_0=x$ (or $B_0=x$). For ease of notation, we write $\mathbb{P}=\mathbb{P}_{0}.$ By the mean value theorem, there exists some $\xi\in(0,s)$ such that
 $$sH(s)=s-1+f(1-s)=\Big(1-\sum^{\infty}_{k=1}kp_k(1-\xi)^{k-1}\Big)s\leq s.$$
 Thus, we have $H(s)\in[0,1]$ for $s\in[0,1].$
\begin{lemma} \label{454t67ji1}For $x>0$, under $\mathbb{P}_x$, the process
\begin{align}
Y_n:=\Big[\prod^n_{j=1}\big(1-H(u(W_j))\big)\Big]u(W_n), n\geq1\nonumber
\end{align}
is a bounded martingale with respect to the natural filtration generated by the random
walk $\{W_n\}_{n\geq1}$.
\end{lemma}
\begin{proof}Since $Q(s)=s-sH(s)$, by (\ref{45kopq1}), we have
\begin{align}\label{45hy789q1}
u(x)=Q(\mathbf{E}[u(x-X)])&=\mathbf{E}[u(x-X)]-\mathbf{E}[u(x-X)]H(\mathbf{E}[u(x-X)])\cr
&=\mathbb{E}_x[u(W_1)]-\mathbb{E}_x[u(W_1)]H(\mathbb{E}_x[u(W_1)]).
\end{align}
One can replace \cite[(16)]{lalley2015} with (\ref{45hy789q1}) in the proof of \cite[Propostion 6]{lalley2015} to obtain Lemma \ref{454t67ji1}.
\end{proof}
Since $u$ is non-increasing, by the Bolzano-Weierstrass theorem \cite[Theorem 3.4.8]{Bartle2000}, for any $y\geq0$ there exist sequences $y_k\to\infty$ such that
\begin{align}
\lim_{k\to\infty}\frac{u(y_k+y/\sqrt{u(y_k)})}{u(y_k)}\in[0,1]~\text{exists}.\nonumber
\end{align}
By a diagonalization argument, we can find a sequence $\{x_k\}_{k\geq1}$ with $\lim_{k\to\infty}x_k=\infty$ such that for all rational $y\geq0$,
\begin{align}\label{4875623re}
\phi(y):=\lim_{k\to\infty}\frac{u(x_k+y/\sqrt{u(x_k)})}{u(x_k)}.
\end{align}

The following lemma comes from \cite[Proposition 8]{lalley2015} or \cite[Proposition 3.2]{houbernoulli}.
\begin{lemma}\label{4656tuh1} Assume $\mathbb{E}[X^2]<\infty$ and $\sigma^2<\infty.$ For any sequence $\{x_k\}_{k\geq1}$ with $\lim_{k\to\infty}x_k=\infty$ such that (\ref{4875623re}) holds for all rational $y\geq0$, the limit in (\ref{4875623re}) also exists for all $y\geq0$. Therefore, we can use (\ref{4875623re}) to define $\phi(y)$ for all $y\geq0$. Moreover, $\phi$ is continuous and the convergence (\ref{4875623re}) holds uniformly for $y$ in any compact interval of $[0,\infty)$.
\end{lemma}
\begin{remark}
One can follow the steps of  \cite[Propostion 3.2]{houbernoulli} and \cite[Proposition 8]{lalley2015} to prove Lemma \ref{4656tuh1}. Since the proofs mainly use Donsker's principle and the fact that $H(u)\sim\sigma^2u/2$ as $u\to0+$ (see (\ref{45kio12}) below), we point out that the assumptions $\mathbb{E}[X^2]<\infty$ and $\sigma^2<\infty$ are enough for Lemma \ref{4656tuh1}.
\end{remark}

The following proposition shows that any subsequential limit $\phi$ satisfies the following Feynman-Kac formula, which weakens assumptions $\sum^{\infty}_{k=0}k^3p_k<\infty$ and $\mathbb{E}[|X|^{4+\varepsilon}]<\infty$ in Lalley and Shao \cite[Propostion 9]{lalley2015}. The key point is that instead of using the moment method of the overshoot distribution like \cite{lalley2015}, we take advantage of the central limit theorem of the maximum of a random walk. Set $\tau^{BM}:=\inf\{t\geq0:B_t=0\}$. Recall that $X^+=\max\{X,0\}$, $\sigma^2=\mathbf{E}[|Z_1|^2]-1$ and $\eta^2=\mathbb{E}[X^2].$
\begin{proposition}\label{465hyu1m} Assume $\sigma^2<\infty$, $\eta^2<\infty$ and $\mathbb{E}[(X^+)^4]<\infty$. Then, any subsequential limit $\phi(y)$, $y\geq0$
satisfies
\begin{align}\label{345454t63}
\phi(y)=\mathbb{E}_{y/\eta}\Big[\exp\Big\{-\frac{\sigma^2}{2}\int^{\tau^{BM}}_0\phi(\eta B_t)dt\Big\}\Big].
\end{align}
\end{proposition}
\begin{proof}
 Since $\frac{e^{-x}-1+x}{x^2}\in(0,1/2)$ for $x>0$ and $\mathbf{E}[|Z_1|^2]<\infty$, by the dominated convergence theorem, we have
\begin{align}\label{4fgt6ye3}
\lim_{s\to0+}\frac{\mathbf{E}[e^{-s|Z_1|}]-1+\mathbf{E}[|Z_1|]s}{s^2}
=\mathbf{E}\Big[\lim_{s\to0+}\frac{e^{-s|Z_1|}-1+s|Z_1|}{s^2}\Big]=\frac{\mathbf{E}[|Z_1|^2]}{2},
\end{align}
where the last equality follows from the fact that $e^{-x}=1-x+x^2/2+o(x^2)$ as~$x\to0$. Let $s=-\log(1-r)$, it follows that
\begin{align}\label{45kio12}
\lim_{r\to0+}\frac{H(r)}{r}
&=\lim_{r\to0+}\frac{\mathbf{E}[(1-r)^{|Z_1|}+r-1]}{r^2}\cr
&=\lim_{r\to0+}\frac{\mathbf{E}[e^{{|Z_1|}\log(1-r)}+r-1]}{r^2}\cr
&=\lim_{s\to0+}\frac{\mathbf{E}[e^{-s{|Z_1|}}-e^{-s}]}{(1-e^{-s})^2}\cr
&=\lim_{s\to0+}\frac{\mathbf{E}[e^{-s{|Z_1|}}-1+s]}{s^2}+\lim_{s\to0+}\frac{1-s-e^{-s}}{s^2}\cr
&=\frac{\sigma^2}{2},
\end{align}
 where the last equality follows from (\ref{4fgt6ye3}). Let $\{x_k\}_{k\geq1}$ be the sequence defined in Lemma \ref{4656tuh1}. Set $z_k=x_k+\lambda u(x_k)^{-1/2}$ with $\lambda>0$. For $k\geq1$ and $j\geq0$, put
\begin{align}\label{34kpq21}
g_{k,j}&:=\left[\log[1-H(u(W_j+z_k))]+\frac{\sigma^2}{2}u(W_j+z_k)\right]/u(x_k)\cr
&=\left[\frac{\log[1-H(u(W_j+z_k))]}{H(u(W_j+z_k))}\frac{H(u(W_j+z_k))}{u(W_j+z_k)}+\frac{\sigma^2}{2}\right]\frac{u(W_j+z_k)}{u(x_k)}.
\end{align}
By (\ref{45kio12}), we have
\begin{align}
\lim_{r\to0+}\frac{\log[1-H(r)]}{H(r)}\frac{H(r)}{r}=-\frac{\sigma^2}{2}.
\end{align}
Thus, for any $\varepsilon>0$, there exists $\delta>0$ such that for $r\in(0,\delta)$,
\begin{align}\label{4hyhyhke3}
\left|\frac{\log[1-H(r)]}{H(r)}\frac{H(r)}{r}+\frac{\sigma^2}{2}\right|\leq \varepsilon.
\end{align}
Since $\lim_{k\to\infty}u(x_k)=0,$ there exists a non-random integer $N>0$ such that for $k>N$ and $j$ satisfying $W_j\geq 0$, it follows that
\begin{align}
0\leq u(W_j+z_k)=u(W_j+x_k+\lambda u(x_k)^{-1/2})\leq u(x_k)\leq \delta.
\end{align}
By (\ref{4hyhyhke3}), for $k>N$ and $j$ satisfying $W_j\geq 0$, we have
\begin{align}\label{45kopw}
|g_{k,j}|&\leq \Bigg{|}\frac{\log[1-H(u(W_j+z_k))]}{H(u(W_j+z_k))}\frac{H(u(W_j+z_k))}{u(W_j+z_k)}+\frac{\sigma^2}{2}\Bigg{|}\cr
&\leq\varepsilon,
\end{align}
where the first inequality follows from $0\leq u(W_j+z_k)\leq u(x_k)$.
\par

By the law of the iterated logarithm \cite[Theorem 5.17]{peteryuval}, one can set $\tau:=\min\{i\geq 0:W_i\leq 0\}<\infty$. This definition implies that
\begin{align}
W_{i}>0,~0\leq i\leq {\tau-1}~\text{and}~W_{\tau}\leq0.\nonumber
\end{align}
Let
\begin{align}\label{56hujuju78w2}
t_k:=\lfloor h_ku^{-1}(x_k)\rfloor~\text{with}~ h_k:=1/(\mathbb{E}[(X^+)^4\ind_{\{X\geq \lambda u(x_k)^{-1/2}\}}])^{1/2},
\end{align}
where $\lfloor x\rfloor$ stands for the largest integer not exceeding $x$. Note that if $\text{ess sup} X<\infty$, then $h_k=t_k=\infty$ for large $k$. Since (\ref{345454t63}) holds trivially for $y=0$, in the remainder of this proof, we assume $y>0$. Similarly to \cite[(22)]{lalley2015}, by Lemma \ref{454t67ji1} and Doob's optional stopping time theorem, we have
\begin{align}\label{345yu71}
&\frac{1}{u(x_k)}u(x_k+yu(x_k)^{-1/2}+\lambda u(x_k)^{-1/2})\cr
&=\mathbb{E}_{yu(x_k)^{-1/2}}\Bigg[\exp\Big\{\sum^{\tau\wedge t_k}_{j=1}\log[1-H(u(W_j+z_k))]\Big\}\frac{u(W_{\tau\wedge t_k}+z_k)}{u(x_k)}\Bigg]\cr
&=\mathbb{E}_{yu(x_k)^{-1/2}}\Bigg[\exp\Big\{-\frac{\sigma^2}{2}\sum^{\tau\wedge t_k}_{j=1}u(W_j+z_k)+u(x_k)\sum^{\tau\wedge t_k}_{j=1} g_{k,j}\Big\}\frac{u(W_{\tau\wedge t_k}+z_k)}{u(x_k)}\Bigg],
\end{align}
where the last equality follows from (\ref{34kpq21}).
\par

In next, we are going to show that
\begin{align}\label{4656juo9}
\lim_{\lambda\to 0+}\limsup_{k\to\infty}\mathbb{E}_{yu(x_k)^{-1/2}}\Big[\Big{|}\frac{u(W_{\tau\wedge t_k}+z_k)}{u(x_k)}-1\Big{|}\Big]=0.
\end{align}
We first consider the case that $\text{ess sup} X=\infty$, which implies that $\{h_k\}_{k\geq1},\{t_k\}_{k\geq1}$ are finite sequences.
Since on the event $\{\tau>t_k\}$, we have $W_{t_k}>0$. Thus,
\begin{align}
\Big{|}\frac{u(W_{\tau\wedge t_k}+z_k)}{u(x_k)}-1\Big{|}\leq \frac{u(W_{t_k}+z_k)}{u(x_k)}+1\leq 2.\nonumber
\end{align}
Above yields that
\begin{align}\label{4frgt67u7}
&\mathbb{E}_{yu(x_k)^{-1/2}}\Big[\Big{|}\frac{u(W_{\tau\wedge t_k}+z_k)}{u(x_k)}-1\Big{|}\ind_{\{\tau>t_k\}}\Big]\cr
&\leq 2\mathbb{P}_{yu(x_k)^{-1/2}}\left(\tau> t_k\right)\cr
&= 2\mathbb{P}_{yu(x_k)^{-1/2}}\Big(\inf_{i\leq t_k}W_i\geq 0\Big).
\end{align}
Recall that $\eta^2=\mathbb{E}[X^2]$. Let $\{S_n\}_{n\geq0}$ be a random walk with increment $X$. Since $\lim_{k\to\infty}h_k=\infty$ and $\lim_{k\to\infty}u(x_k)=0$, we have
\begin{align}\label{54hyui2w}
\lim_{k\to\infty}(h_ku^{-1}(x_k)-1)u(x_k)=\lim_{k\to
\infty}h_k-u(x_k)=\infty.
\end{align}
By \cite[Theorem 5.25]{peteryuval}, we have
\begin{align}\label{9034kgt}
\lim_{k\to\infty}\mathbb{P}_{yu(x_k)^{-1/2}}\left(\inf_{i\leq t_k}W_i\geq 0\right)
&=\lim_{k\to\infty}\mathbb{P}_{0}\left(\max_{i\leq t_k}S_i\leq yu(x_k)^{-1/2}\right)\cr
&=\lim_{k\to\infty}\mathbb{P}_{0}\left(\frac{\max_{i\leq t_k}S_i}{\sqrt{t_k}}\leq \frac{yu(x_k)^{-1/2}}{\sqrt {\lfloor h_ku^{-1}(x_k)\rfloor}}\right)\cr
&\leq\lim_{k\to\infty}\mathbb{P}_{0}\left(\frac{\max_{i\leq t_k}S_i}{\sqrt{t_k}}\leq \frac{y}{\sqrt { (h_ku^{-1}(x_k)-1)u(x_k)}}\right)\cr
&=0,
\end{align}
where the last equality has used (\ref{54hyui2w}). This, combined with (\ref{4frgt67u7}), implies that
\begin{align}\label{45ghy56hy1}
\lim_{k\to\infty}\mathbb{E}_{yu(x_k)^{-1/2}}\Big[\Big{|}\frac{u(W_{\tau\wedge t_k}+z_k)}{u(x_k)}-1\Big{|}\ind_{\{\tau>t_k\}}\Big]=0.
\end{align}
\par
Now, we are going to deal with ${\tau\leq t_k}$. Put
$$A:=\{W_{\tau}\geq -\lambda u(x_k)^{-1/2}\}.$$
Since $W_{\tau}\leq0$, on the event $A\cap\{\tau\leq t_k\}$, we have
\begin{align}
\frac{u(x_k+\lambda u(x_k)^{-1/2})}{u(x_k)}\leq\frac{u(W_{\tau\wedge t_k}+z_k)}{u(x_k)}\leq 1.\nonumber
\end{align}
This, combined with Lemma \ref{4656tuh1}, implies that
\begin{align}\label{5657juol5}
\lim_{\lambda\to 0+}\limsup_{k\to\infty}\mathbb{E}_{yu(x_k)^{-1/2}}\Big[\Big{|}\frac{u(W_{\tau\wedge t_k}+z_k)}{u(x_k)}-1\Big{|}\ind_{A\cap\{\tau\leq t_k\}}\Big]=0.
\end{align}
We proceed to deal with the event $A^c\cap\{\tau\leq t_k\}$. Since $0\leq u(x)\leq 1$, we have
\begin{align}\label{4512psm}
\Big{|}\frac{u(W_{\tau\wedge t_k}+z_k)}{u(x_k)}-1\Big{|}\leq \frac{1}{u(x_k)}.
\end{align}
By the Markov inequality, it follows that
\begin{align}\label{56yghy67y}
\mathbb{P}(\exists 1\leq i\leq t_k, X_i\geq \lambda u(x_k)^{-1/2})&\leq  \frac{t_k\mathbb{E}[(X^+)^4\ind_{\{X\geq \lambda u(x_k)^{-1/2}\}}]}{\lambda^4u(x_k)^{-2}}\cr
&\leq \frac{h_k\mathbb{E}[(X^+)^4\ind_{\{X\geq \lambda u(x_k)^{-1/2}\}}]}{\lambda^4u(x_k)^{-1}}.
\end{align}
Observe that since $W_{\tau-1}\geq 0$, if $A^c=\{W_{\tau}<-\lambda u(x_k)^{-1/2}\}$ and $\tau\leq t_k$ holds simultaneously, then there must exists $1\leq i\leq t_k$ such that $-X_i\leq -\lambda u(x_k)^{-1/2}$.
This, combined with (\ref{4512psm}), implies that
\begin{align}\label{567opw2}
&\mathbb{E}_{yu(x_k)^{-1/2}}\Big[\Big{|}\frac{u(W_{\tau\wedge t_k}+z_k)}{u(x_k)}-1\Big{|}\ind_{A^c\cap\{\tau\leq t_k\}}\Big]\cr
&\leq\frac{1}{u(x_k)}\mathbb{P}_{yu(x_k)^{-1/2}}({A^c}, \tau\leq t_k)\cr
&\leq\frac{1}{u(x_k)}\mathbb{P}(\exists 1\leq i\leq t_k, -X_i\leq -\lambda u(x_k)^{-1/2})\cr
&=\frac{1}{u(x_k)}\mathbb{P}(\exists 1\leq i\leq t_k, X_i\geq \lambda u(x_k)^{-1/2})\cr
&\leq \frac{1}{\lambda^4}(\mathbb{E}[(X^+)^4\ind_{\{X\geq \lambda u(x_k)^{-1/2}\}}])^{1/2},
\end{align}
where the last inequality follows from (\ref{56hujuju78w2}) and (\ref{56yghy67y}). Since $\mathbb{E}[(X^+)^4]<\infty$, by the dominated convergence theorem, (\ref{567opw2}) yields that
\begin{align}\label{45juo91}
\lim_{k\to\infty}\mathbb{E}_{yu(x_k)^{-1/2}}\Big[\Big{|}\frac{u(W_{\tau\wedge t_k}+z_k)}{u(x_k)}-1\Big{|}\ind_{A^c\cap\{\tau\leq t_k\}}\Big]=0.
\end{align}
This, combined with (\ref{45ghy56hy1}) and (\ref{5657juol5}), concludes (\ref{4656juo9}).
\par
Now, we are going to prove (\ref{4656juo9}) under the case that $\text{ess sup} X<\infty$. For $k$ large enough such that  $\lambda u(x_k)^{-1/2}\geq \text{ess sup} X$, since $W_{\tau-1}\geq0$, the event $A=\{W_{\tau}\geq -\lambda u(x_k)^{-1/2}\}$ happens almost surely. This, combined with the fact that $t_k=\infty$ for large $k$, yields that
\begin{align}
&\lim_{\lambda\to 0+}\limsup_{k\to\infty}\mathbb{E}_{yu(x_k)^{-1/2}}\Big[\Big{|}\frac{u(W_{\tau\wedge t_k}+z_k)}{u(x_k)}-1\Big{|}\Big]\cr
&=\lim_{\lambda\to 0+}\limsup_{k\to\infty}\mathbb{E}_{yu(x_k)^{-1/2}}\Big[\Big{|}\frac{u(W_{\tau\wedge t_k}+z_k)}{u(x_k)}-1\Big{|}\ind_{A\cap\{\tau\leq t_k\}}\Big]=0,
\end{align}
where the last equality follows from (\ref{5657juol5}). We have finished the prove of (\ref{4656juo9}).
\par
By Donsker's invariance principle, for any $z\geq0$, we have
\begin{align}\label{456unjji89de}
\lim_{x\to+\infty}\mathbb{P}_{{y}/{\sqrt {u(x)}}}(\tau u(x)\leq z)=\mathbb{P}_{y}(\tau^{BM}\leq z).
\end{align}
Thus, under the probability $\mathbb{P}_{yu(x_k)^{-1/2}}$, we have
\begin{align}\label{456jiop1}
&\sum^{\tau\wedge t_k}_{j=1}u(W_j+z_k)\cr
&=\sum_{i=u(x_k),2u(x_k),...,(\tau\wedge t_k) u(x_k)}\frac{1}{u(x_k)}u\left(\frac{\sqrt{u(x_k)}W_{\frac{i}{u(x_k)}}}{\sqrt{u(x_k)}}+\lambda u(x_k)^{-1/2}+x_k\right)u(x_k)\cr
&\overset{\text{Law}}\to\int^{\tau^{BM}}_{0}\phi(\eta B_t+\lambda)dt~\text{as}~k\to\infty.
\end{align}
Intuitively, the convergence above is a consequence of Lemma \ref{4656tuh1}, Donsker's invariance principle and Riemann integral approximation. The proof of (\ref{456jiop1}) is almost the same as \cite[Lemma 3.3]{houbernoulli}. We omit the details here. It is simple to see that
\begin{align}\label{4tgty6h6}
\Big{|}u(x_k)\sum^{\tau \wedge t_k}_{j=1} g_{k,j}\Big{|}\leq u(x_k)\sum^{\tau -1}_{j=1} \Big{|}g_{k,j}\Big{|}+\Big{|}\log[1-H(u(W_{\tau\wedge t_k}+z_k))]+\frac{\sigma^2}{2}u(W_{\tau\wedge t_k}+z_k)\Big{|}.
\end{align}
By (\ref{45kopw}) and (\ref{456unjji89de}), under the probability $\mathbb{P}_{yu(x_k)^{-1/2}}$, the first term on the r.h.s. of (\ref{4tgty6h6}) converges in probability to $0$ as $k\to\infty$. Since $\lim_{k\to\infty}u(x_k)=0$, from (\ref{4656juo9}), we have
\begin{align}
&\lim_{k\to\infty}\mathbb{E}_{yu(x_k)^{-1/2}}[u(W_{\tau\wedge t_k}+z_k)]=0.\nonumber
\end{align}
This, combined with (\ref{45kio12}), yields that under the probability $\mathbb{P}_{yu(x_k)^{-1/2}}$, the second term on the r.h.s. of (\ref{4tgty6h6}) converges in probability to $0$ as $k\to\infty$. Hence, under the probability $\mathbb{P}_{yu(x_k)^{-1/2}}$, we have as $k\to\infty$,
\begin{align}\label{4546hyu91}
u(x_k)\sum^{\tau \wedge t_k}_{j=1} g_{k,j}\overset{\text{Law}}\to0.
\end{align}
By (\ref{34kpq21}), (\ref{456jiop1}) and (\ref{4546hyu91}), we have
\begin{align}\label{ervguy345th}
\sum^{\tau\wedge t_k}_{j=1}\log[1-H(u(W_j+z_k))]&=-\frac{\sigma^2}{2}\sum^{\tau\wedge t_k}_{j=1}u(W_j+z_k)+u(x_k)\sum^{\tau\wedge t_k}_{j=1} g_{k,j}\cr
&\overset{\text{Law}}\to-\frac{\sigma^2}{2}\int^{\tau^{BM}}_{0}\phi(\eta B_t+\lambda)dt~\text{as}~k\to\infty.
\end{align}
\par
From (\ref{345yu71}), we have
\begin{align}\label{45gtyty781}
&\frac{1}{u(x_k)}u(x_k+yu(x_k)^{-1/2}+\lambda u(x_k)^{-1/2})\cr
&=\mathbb{E}_{yu(x_k)^{-1/2}}\Bigg[\exp\Big\{\sum^{\tau\wedge t_k}_{j=1}\log[1-H(u(W_j+z_k))]\Big\}\Bigg]+\cr
&~~~~\mathbb{E}_{yu(x_k)^{-1/2}}\Bigg[\exp\Big\{\sum^{\tau\wedge t_k}_{j=1}\log[1-H(u(W_j+z_k))]\Big\}\Big(\frac{u(W_{\tau\wedge t_k}+z_k)}{u(x_k)}-1\Big)\Bigg].
\end{align}
Since $\sum^{\tau\wedge t_k}_{j=1}\log[1-H(u(W_j+z_k))]\leq 0$, by (\ref{4656juo9}), the second term on the r.h.s. of (\ref{45gtyty781}) converges to $0$ as $k\to\infty$ and $\lambda\to 0+$. Taking the limits in $k\to\infty$ first and then $\lambda\to0+$ in (\ref{45gtyty781}), by Lemma \ref{4656tuh1} and (\ref{ervguy345th}), we have
\begin{align}
\phi(y)=\mathbb{E}_{y/\eta}\Big[\exp\Big\{-\frac{\sigma^2}{2}\int^{\tau^{BM}}_0\phi(\eta B_t)dt\Big\}\Big].\nonumber
\end{align}
We have finished the proof of the proposition.
\end{proof}
Since there is only one solution to equation (\ref{345454t63}) (see \cite[Corollary 12]{lalley2015} and \cite[Corollary 3.5]{houbernoulli}), by Lemma \ref{4656tuh1} and Proposition \ref{465hyu1m}, we have the following corollary.
\begin{corollary} \label{34gtgt3}Assume $\sigma^2<\infty$, $\eta^2<\infty$ and $\mathbb{E}[(X^+)^4]<\infty$. Then, for any $y\geq0$,
$$
\lim_{x\to\infty}\frac{u(x+y/\sqrt{u(x)})}{u(x)}=\Big(1+\frac{\sigma y}{\sqrt 6\eta}\Big)^{-2}=\phi(y).
$$

\end{corollary}
{\textbf{Proof of Theorem \ref{thdimmaxplace}.}} Now, with Corollary \ref{34gtgt3} in hand, Theorem \ref{thdimmaxplace} follows by using exactly the same steps as those of \cite[p. 646-647]{houbernoulli} or \cite[p. 85-86]{lalley2015}.

Theorem \ref{thdimmaxplace} assumes the offspring has finite variance. If the offspring law has infinite variance, then we have the following results.
\begin{corollary}\label{4gtty6y675} Suppose that $\lim_{n\to\infty}n^{1+\beta}\mathbf{P}(|Z_1|\geq n)=\kappa(\beta)\in(0,\infty)$ with $\beta\in(0,1)$, $\eta^2<\infty$ and $\mathbb{E}[(X^+)^{2+\frac{2}{\beta}}]<\infty$. Then,
\begin{align}
\lim_{x\to\infty}x^{\frac{2}{\beta}}\mathbf{P}(M\geq x)=\left(\frac{(\beta+2)\eta^2}{\beta\kappa(\beta)\Gamma(1-\beta)}\right)^{\frac{1}{\beta}}.
\end{align}
\end{corollary}
\begin{proof}From \cite[Lemma 3.1]{houbernoulli}, we have
$\lim_{r\to0+}\frac{H(r)}{r^{\beta}}=\frac{\kappa(\beta)\Gamma(1-\beta)}{\beta}.$ Now, the proofs are almost the same as those of Theorem \ref{thdimmaxplace}.
The main differences are to replace $\sqrt{u(x)}$ with $u(x)^{-\beta/2}$ and replace $\frac{\sigma^2}{2}$ with $\frac{\kappa(\beta)\Gamma(1-\beta)}{\beta}$.
\end{proof}

The following lemma considers the tail probability of the maximum of the branching random walk at time $n$, which is crucial in studying the empty ball problem. The lemma is similar to that of the branching L\'evy process \cite[Lemma 3.5]{hou2024}. However, different from \cite[Lemma 3.5]{hou2024}, our bounds are universal w.r.t. $n$ and assumptions are weaker. Recall that $S_u$ is the position of particle $u$.
\begin{lemma}\label{566u9g13q} Let $M_n:=\max_{u\in Z_n}S_u.$ (i) Suppose that $\lim_{n\to\infty}n^{1+\beta}\mathbf{P}(|Z_1|\geq n)=\kappa(\beta)\in(0,\infty)$ with $\beta\in(0,1)$ and $\mathbb{E}[(X^+)^{2+\frac{2}{\beta}}]<\infty$. Let $C_{\beta}>0$ be the constant defined in (\ref{454tyhgty5}) below. Then, for any $r,x>0$ and $n\geq0$,
\begin{align}\label{45kop9}
n^{\frac{1}{\beta}}\mathbf{P}(M_{\lfloor nr\rfloor}>x\sqrt n)\leq C_{\beta} \frac{r}{x^{2+\frac{2}{\beta}}}.
\end{align}
(ii) Suppose that $\mathbf{E}[|Z_1|^2]<\infty$ and $\mathbb{E}[(X^+)^{4}]<\infty$. Let $C_1>0$ be the constant defined in (\ref{454tyhgty5}) with $\beta=1$. For any $r,x>0$ and $n\geq0$, we have
\begin{align}\label{45k00fd3d49}
n\mathbf{P}(M_{\lfloor nr\rfloor}>x\sqrt n)\leq C_{1} \frac{r}{x^{4}}.
\end{align}
\end{lemma}
\begin{proof}  For $y\in\mathbb{R}$ and $n\geq0$, set
 \begin{align}\label{56hyy681}
 v_n(y):=\mathbf{P}(Z_n((y,\infty))>0)=\mathbf{P}(M_n>y).
 \end{align}
Recall that $\{W_i\}_{i\geq0}$ is a random walk with step size following the distribution of $-X$ and $H(s)=[s-1+\sum^{\infty}_{i=0}p_i(1-s)^i]/s$. We make the convention that the empty product equals $1$. Fix $m\geq0$. By \cite[Proposition 14]{lalley2015}, we have
$$Z^{(m)}_k:=v_{m-k}(W_k)\prod^{k-1}_{j=1}\big[1-H(v_{m-j}(W_j))\big], 0\leq k\leq m,$$
is a martingale w.r.t. the natural filtration of $\{W_k\}_{0\leq k\leq m}$. Thus, for any $y$, we have
\begin{align}\label{4rrgt54pl2}
v_m(y)=\mathbb{E}_y\Big[v_{m-k}(W_k)\prod^{k-1}_{j=1}\big[1-H(v_{m-j}(W_j))\big]\Big].
\end{align}
Set
$$\gamma:=\inf\{i\geq0:W_i\leq x\sqrt n/2\}.$$
For $r,x>0$ and $n\geq 0$, let $y=x\sqrt n$, $m=\lfloor nr\rfloor$ and $k=m\wedge\gamma$ in (\ref{4rrgt54pl2}). By Doob's bounded stopping time theorem, it yields that
\begin{align}\label{45lop1}
v_{\lfloor nr\rfloor}(x\sqrt n)=\mathbb{E}_{x\sqrt n}\Big[v_{\lfloor nr\rfloor-{\lfloor nr\rfloor}\wedge\gamma}(W_{\lfloor nr\rfloor\wedge\gamma})
\prod^{\lfloor nr\rfloor\wedge\gamma-1}_{j=1}\big[1-H(v_{\lfloor nr\rfloor-j}(W_j))\big]\Big].
\end{align}
 Since $0\leq H(s)\leq 1$ (see the line above Lemma \ref{454t67ji1}), by (\ref{45lop1}), it follows that
\begin{align}\label{978ki4gh1}
&v_{\lfloor nr\rfloor}(x\sqrt n)\cr
&\leq\mathbb{E}_{x\sqrt n}\Big[v_{\lfloor nr\rfloor-{\lfloor nr\rfloor}\wedge\gamma}(W_{\lfloor nr\rfloor\wedge\gamma})\Big]\cr
&=\mathbb{E}_{x\sqrt n}\Big[v_{\lfloor nr\rfloor-{\lfloor nr\rfloor}\wedge\gamma}(W_{\lfloor nr\rfloor\wedge\gamma})\ind_{\{\gamma\leq {\lfloor nr\rfloor}\}}\Big]\cr
&=\mathbb{E}_{x\sqrt n}\Big[v_{\lfloor nr\rfloor-{\lfloor nr\rfloor}\wedge\gamma}(W_{\gamma})\ind_{\{\gamma\leq {\lfloor nr\rfloor}\}}\left(\ind_{\{W_{\gamma }/\sqrt n\geq x/4\}}+\ind_{\{W_{\gamma }/\sqrt n< x/4\}}\right)\Big],
\end{align}
where the first equality follows from the fact that if $\gamma>{\lfloor nr\rfloor}$, then $W_{\lfloor nr\rfloor}>x\sqrt n/2>0$ and
\begin{align}
v_{\lfloor nr\rfloor-{\lfloor nr\rfloor}\wedge\gamma}(W_{\lfloor nr\rfloor\wedge\gamma})\leq\mathbf{P}(M_0\geq x\sqrt n/2)=0.\nonumber
\end{align}
\par
Without loss of generality, we only prove (i) of Lemma \ref{566u9g13q}. By Corollary \ref{4gtty6y675}, there exists a constant $c_1$ such that for any $x>0$,
$$\mathbf{P}(M\geq x)\leq c_1x^{-\frac{2}{\beta}}.$$
Note that for any $z>x\sqrt n/4$ and $n\geq 1$, we have
\begin{align}\label{rtmjkio49b}
n^{\frac{1}{\beta}}v_n(z)=n^{\frac{1}{\beta}}\mathbf{P}(M_n\geq z)\leq n^{\frac{1}{\beta}}\mathbf{P}(M\geq x\sqrt n/4)\leq c_14^{\frac{2}{\beta}}x^{-\frac{2}{\beta}}.
\end{align}
Plugging (\ref{rtmjkio49b}) into (\ref{978ki4gh1}) yields that
\begin{align}\label{5657hyuol}
n^{\frac{1}{\beta}}v_{\lfloor nr\rfloor}(x\sqrt n)\leq c_14^{\frac{2}{\beta}}x^{-\frac{2}{\beta}}\mathbb{P}_{x\sqrt n}(\gamma\leq \lfloor nr\rfloor)+n^{\frac{1}{\beta}}\mathbb{P}_{x\sqrt n}(\gamma\leq \lfloor nr\rfloor,W_{\gamma }\leq x\sqrt n/4).
\end{align}
Recall that $\{S_i\}_{i\geq0}=\{-W_i\}_{i\geq0}$ is the random walk with step size $X$. By Doob's inequality, it yields that
\begin{align}\label{678jo97u}
\mathbb{P}_{x\sqrt n}(\gamma\leq \lfloor nr\rfloor)&=\mathbb{P}_{x\sqrt n}\Big(\min_{0\leq i\leq \lfloor nr\rfloor}W_i\leq x\sqrt n/2\Big)\cr
&=\mathbb{P}_{0}\Big(\min_{0\leq i\leq \lfloor nr\rfloor}W_i\leq -x\sqrt n/2\Big)\cr
&=\mathbb{P}_{0}\Big(\max_{0\leq i\leq \lfloor nr\rfloor}S_i\geq x\sqrt n/2\Big)\cr
&\leq \frac{4}{nx^2}\mathbb{E}\big[\max_{0\leq i\leq \lfloor nr\rfloor}S^2_i\big]\cr
&\leq \frac{16}{nx^2}\mathbb{E}\big[S^2_{\lfloor nr\rfloor}\big]\leq\frac{16\mathbb{E}[X^2]}{x^2}r.
\end{align}
Observe that by the definition of $\gamma$, we have $W_{\gamma-1}\geq x\sqrt n/2.$ Thus, if $\gamma\leq \lfloor nr\rfloor$ and $W_{\gamma}\leq x\sqrt n/4$, then there must exist $i\in {1,2...,\lfloor nr\rfloor}$ (recall that the step size of $\{W_i\}_{i\geq0}$ is $-X$) such that
 $$-X_i\leq x\sqrt n/4-x\sqrt n/2=-x\sqrt n/4.$$
This, together with the Markov inequality, implies that
\begin{align}\label{565hyu634}
&n^{\frac{1}{\beta}}\mathbb{P}_{x\sqrt n}(\gamma\leq \lfloor nr\rfloor,W_{\gamma}\leq x\sqrt n/4)\cr
&\leq n^{\frac{1}{\beta}}\mathbb{P}_{0}(\exists i\in {1,2...,\lfloor nr\rfloor}~\text{s.t.}-X_i\leq -x\sqrt n/4)\cr
&\leq n^{\frac{1}{\beta}}nr\mathbb{P}( X\geq x\sqrt n/4).\cr
&\leq (\frac{4}{x})^{2+\frac{2}{\beta}}\mathbb{E}[(X^+)^{2+\frac{2}{\beta}}]r.
\end{align}
Plugging (\ref{565hyu634}) and (\ref{678jo97u}) into (\ref{5657hyuol}) yields that for any $x,r>0$ and $n\geq0$,
\begin{align}\label{454tyhgty5}
n^{\frac{1}{\beta}}\mathbf{P}(M_{\lfloor nr\rfloor}\geq x\sqrt n)&=n^{\frac{1}{\beta}}v_{\lfloor nr\rfloor}(x\sqrt n)\cr
&\leq\left(16c_14^{\frac{2}{\beta}}\mathbb{E}[X^2]+4^{2+\frac{2}{\beta}}\mathbb{E}[(X^+)^{2+\frac{2}{\beta}}]\right)\frac{1}{x^{2+\frac{2}{\beta}}}r\cr
&=:C_{\beta}\frac{r}{x^{2+\frac{2}{\beta}}} .
\end{align}
We have finished the proof of the lemma.
\end{proof}

\section{Proof of Theorem \ref{thdim2}: Empty Ball With Finite Variance}\label{45gttghy9hn3r}
In this section, we shall obtain asymptotic behaviours of the empty ball of a 2-dimensional branching random walk whose offspring law has finite variance. To this end, we first present a lemma regarding the continuity of the local extinction probability of the super-Brownian motion.
\par
\begin{lemma}\label{4yu7ui89w24}Let $\{X_t\}_{t\geq0}$ be a two dimensional critical super-Brownian motion. Fix $r>0$ and $x\in\mathbb{R}^2$. Then,
$
\mathbb{P}_{\delta_x}(X_t(B(r))=0)
$ is continuous with respect to $t$ on $(0,\infty)$.
\end{lemma}
\begin{proof}
It is well-known (e.g. \cite[p. 2]{ren2021spa}) that if $\phi$ is a bounded nonnegative function, then
\begin{align}\label{4tgyhyu7jt5}
u_{\phi}(s,x):=-\log\mathbb{E}_{\delta_x}\left[e^{-X_{s}(\phi)}\right], s\geq0, x\in \mathbb{R}^2,
\end{align}
is the unique positive solution of the partial differential equation :
 \begin{align}
   \begin{cases}\label{43ty45gtt5y}
   \frac{\partial u(s,x)}{\partial s}=\mathcal{L}u(s,x)-\sigma^2u^2(s,x),\cr
   u(0,x)=\phi(x).
   \end{cases}
   \end{align}
Fix any $t>0$. To prove the lemma, it suffices to show that
$$\mathbb{P}_{\delta_x}(X_{t+s}(B(r))=0)$$
is continuous with respect to $s$ on $(0,\infty).$ Let $\{X^*_t\}_{t\geq0}$ be a super-Bownian motion
 independent of $\{X_t\}_{t\geq0}$, and $\{X^*_t\}_{t\geq0}$ has the same branching mechanism and underlying motion distribution as $\{X_t\}_{t\geq0}$. For any finite measure $\mu$, by (\ref{ujytas}) and (\ref{43ty45gtt5y}), we have
\begin{align}\label{4it46fnr4}
\mathbb{P}_{\mu}(X^*_{t}(B(r))=0)
&=\lim_{\theta\to+\infty}\mathbb{E}_{\mu}[e^{-\theta X^*_{t}(\ind_{B(r)})}]\cr
&=\lim_{\theta\to+\infty}e^{-\int_{\mathbb{R}^2}u_{\theta\ind_{B(r)}}(t,x)\mu(dx)}\cr
&=e^{-\int_{\mathbb{R}^2}\lim_{\theta\to+\infty}u_{\theta\ind_{B(r)}}(t,x)\mu(dx)}\cr
&=e^{-\int_{\mathbb{R}^2}-\log \mathbb{P}_{\delta_x}(X^*_t(B(r))=0)\mu(dx)},
\end{align}
where the third equality follows from L\'evy's monotone convergence theorem and
 \begin{align}
 u_{\theta\ind_{B(r)}}(t,x)=-\log \mathbb{E}_{\delta_x}[e^{-\theta X^*_{t}(\ind_{B(r)})}].
\end{align}
Let
\begin{align}
0\leq \phi(y):=-\log \mathbb{P}_{\delta_y}(X_t^*(B(r))=0))\leq -\log \mathbb{P}_{\delta_0}(X^*_t(\mathbb{R}^2)=0)<\infty.
\end{align}
By the Markov property of $\{X_t\}_{t\geq0}$, it entails that for any $s>0$,
\begin{align}
\mathbb{P}_{\delta_x}(X_{t+s}(B(r))=0)&=\mathbb{P}_{\delta_x}(\mathbb{P}_{\delta_x}(X_{t+s}(B(r))=0|X_s))\cr
&=\mathbb{E}_{\delta_x}[\mathbb{P}_{X_s}(X^*_{t}(B(r))=0)]\cr
&=\mathbb{E}_{\delta_x}\left[e^{X_{s}(\log \mathbb{P}_{\delta_.}(X_t^*(B(r))=0))}\right]\cr
&=\mathbb{E}_{\delta_x}\left[e^{-X_{s}(\phi)}\right]\cr
&=e^{-u_{\phi}(s,x)},\nonumber
\end{align}
where the third equality follows from (\ref{4it46fnr4}).
By (\ref{4tgyhyu7jt5})-(\ref{43ty45gtt5y}), $u_{\phi}(s,x)$ is continuous with respect to $s$ on $(0,\infty)$. Since
\begin{align}
\mathbb{P}_{\delta_x}(X_{t+s}(B(r))=0)=e^{-u_{\phi}(s,x)},
\end{align}
we have $\mathbb{P}_{\delta_x}(X_{t+s}(B(r))=0)$ is continuous with respect to $s$ on $(0,\infty)$.
\end{proof}

Now, we are ready to prove Theorem \ref{thdim2}. The key points of the proof are as follows. We first use the maximum, investigated in Lemma \ref{566u9g13q}, to control the hitting probability $\mathbf{P}(Z_n(\sqrt nB(x,r))>0)$. Then, applying the fact that the scaling limit of the branching random walk is the super-Brownian motion, we obtain Theorem \ref{thdim2}.
\begin{proof}
In \cite[(3.12)]{xzacta}, we proved that
\begin{align}\label{5ty7hy61h}
\limsup_{n\to\infty}\mathbb{P}(R_n\geq \sqrt n r)\leq \exp\left\{\int_{\mathbb{R}^2}\log\mathbb{P}_{\delta_0}(X_1(B(x,r))=0)dx\right\}.
\end{align}
Moreover, from \cite[Theorem 1,2]{xz21}, we have
\begin{align}
\exp\left\{\int_{\mathbb{R}^2}\log\mathbb{P}_{\delta_0}(X_1(B(x,r))=0)dx\right\}\in(0,1).
\end{align}
Hence, it suffices to show the lower bound:
\begin{align}\label{456hu8k3df}
\liminf_{n\to\infty}\mathbb{P}(R_n\geq \sqrt n r)\geq \exp\left\{\int_{\mathbb{R}^2}\log\mathbb{P}_{\delta_0}(X_1(B(x,r))=0)dx\right\}.
\end{align}
\par

To this end, by the branching property, we have
\begin{align}\label{897ojki2fn}
\mathbb{P}(R_n\geq \sqrt n r)
&=\mathbb{P}(Z_n(B(\sqrt nr))=0)\cr
&=\mathbb{E}\left[\prod_{u\in Z_0}\mathbb{P}_{\delta_{S_u}}(Z_n(B(\sqrt nr))=0)\right]\cr
&=\mathbb{E}\left[e^{\sum_{u\in Z_0}\log\mathbb{P}_{\delta_{S_u}}(Z_n(B(\sqrt nr))=0)}\right]\cr
&=\mathbb{E}\left[e^{-\int_{\mathbb{R}^2}-\log \mathbb{P}_{\delta_{x}}(Z_n(B(\sqrt nr))=0)Z_0(dx)}\right]\cr
&=\exp\left\{-\int_{\mathbb{R}^2}\mathbb{P}_{\delta_{x}}(Z_n(B(\sqrt nr))>0)dx\right\}\cr
&=\exp\left\{-\int_{\mathbb{R}^2}n\mathbb{P}_{\delta_0}(Z_n(\sqrt nB(x,r))>0)dx\right\},
\end{align}
where the 5-th equality follows from the Laplace functional formula of the Poisson random measure $Z_0$. Fix $\varepsilon>0$. For $|x|\leq(1+\varepsilon)r$ and $n\geq0$, we have
\begin{align}\label{57hu89k}
n\mathbb{P}_{\delta_0}(Z_n(\sqrt nB(x,r))>0)\leq \sup_{n\geq0}n\mathbf{P}(|Z_n|>0)=:c<\infty,
\end{align}
where the last inequality follows from \cite[p. 19]{athreya72}. Recall that $x^{(i)}$ is the $i$-th component of the vector $x$. Let $\bar{M}_n :=\max\{|S_u|:u\in Z_n\}$, $M'_{n,i}:=\max\{|S^{(i)}_u|:u\in Z_n\}$,
$M''_{n,i}:=\max\{S^{(i)}_u:u\in Z_n\}$ and $M'''_{n,i}:=\min\{S^{(i)}_u:u\in Z_n\}$, where $i=1,2$. By (ii) of Lemma \ref{566u9g13q}, there exists a constant $\rho>0$ such that for $|x|\geq(1+\varepsilon)r$ and $n\geq0$,
\begin{align}\label{5ytghy78i}
&n\mathbf{P}(Z_n(\sqrt nB(x,r))>0)\cr
&\leq n\mathbf{P}(\bar{M}_n\geq \sqrt n(|x|-r))\cr
&\leq \sum^2_{i=1}n\mathbf{P}\left(M'_{n,i}>\frac{(|x|-r)\sqrt {n}}{\sqrt 2}\right)\cr
&\leq \sum^2_{i=1}n\left[\mathbf{P}\left(M''_{n,i}>\frac{(|x|-r)\sqrt {n}}{\sqrt 2}\right)+\mathbf{P}\left(M'''_{n,i}<-\frac{(|x|-r)\sqrt {n}}{\sqrt 2}\right)\right]\cr
&\leq \frac{\rho}{(|x|-r)^4}.
\end{align}
Since
$$c\ind_{B(0,(1+\varepsilon)r)}(x)+\frac{\rho}{(|x|-r)^4}\ind_{B^c(0,(1+\varepsilon)r)}(x), x\in\mathbb{R}^2$$
is integrable on $\mathbb{R}^2$, by Fatou's lemma, we have
\begin{align}\label{5hythtt7j4r}
\liminf_{n\to\infty}\mathbb{P}(R_n\geq \sqrt nr)\geq \exp\left\{-\int_{\mathbb{R}^2}\limsup_{n\to\infty}n\mathbf{P}(Z_n(\sqrt nB(x,r))>0)dx\right\}.
\end{align}
\par
Next, we give an upper bound of $\limsup_{n\to\infty}n\mathbf{P}(Z_n(\sqrt nB(x,r))>0)dx$ by using the super-Brownian motion. Fix $s\in\left(0,\frac{\varepsilon^2}{(r+2\varepsilon)^2}\right)$. Let $\widetilde{{sn}}:=n-\lfloor(1-s)n\rfloor\in[sn,sn+1]$ for short. By the branching property, we have for $n$ large enough,
\begin{align}\label{454t5t89g}
&\mathbf{P}\Big(Z_{n}(\sqrt{n} B(x,r))>0\Big)\cr
&=\mathbf{E}\Big[1-\prod_{v\in Z_{\lfloor(1-s)n\rfloor}}\mathbb{P}_{\delta_{S_v}}(Z_{\widetilde{{sn}}}(\sqrt{ n}B(x,r))=0)\Big]\cr
&=\mathbf{E}\Big[1-\exp\Big\{\int_{\mathbb{R}^2}\log\mathbb{P}_{\delta_{u}}(Z_{\widetilde{{sn}}}(B(\sqrt{n}x,\sqrt {n}r))=0)Z_{\lfloor(1-s)n\rfloor}(du)\Big\}\Big]\cr
&\leq\mathbf{E}\Big[1-\exp\Big\{-2\int_{\mathbb{R}^2} \mathbb{P}_{\delta_{u}}(Z_{\widetilde{{sn}}}(B(\sqrt nx,\sqrt {n}r))>0)Z_{\lfloor(1-s)n\rfloor}(du)\Big\}\Big].
\end{align}
Let us clarify the last inequality of (\ref{454t5t89g}). Since $\lim_{y\to1-}\log(y)/(1-y)=-1$ and
$$\mathbb{P}_{\delta_{u}}(Z_{\widetilde{{sn}}}(\sqrt {n}B(x,r))=0)\geq \mathbb{P}_{\delta_0}(|Z_{\widetilde{{sn}}}|=0),$$
we have
$$\lim_{n\to\infty}\sup_{u\in\mathbb{R}^2}\Bigg{|}\frac{\log \mathbb{P}_{\delta_{u}}(Z_{\widetilde{{sn}}}(\sqrt {n}B(x,r))=0)}{ \mathbb{P}_{\delta_{u}}(Z_{\widetilde{{sn}}}(B(\sqrt nx,\sqrt {n}r))>0)}+1\Bigg{|}=0,$$
which implies that the last inequality of (\ref{454t5t89g}) holds. It is simple to see that
\begin{align}\label{5657hy89w}
&\int_{\mathbb{R}^2} \mathbb{P}_{\delta_{u}}(Z_{\widetilde{{sn}}}(B(\sqrt nx,\sqrt nr))>0)Z_{\lfloor(1-s)n\rfloor}(du)\cr
&=\int_{|u-\sqrt {n}x|\geq \sqrt {n}(r+\varepsilon)} \mathbb{P}_{\delta_{u}}(Z_{\widetilde{{sn}}}(B(\sqrt {n}x,\sqrt {n}r))>0)Z_{\lfloor(1-s)n\rfloor}(du)+\cr
&~~~~\int_{|u-\sqrt {n}x|< \sqrt {n}(r+\varepsilon)} \mathbb{P}_{\delta_{u}}(Z_{\widetilde{{sn}}}(B(\sqrt {n}x,\sqrt {n}r))>0)Z_{\lfloor(1-s)n\rfloor}(du).
\end{align}
For ${|u-\sqrt {n}x|\geq \sqrt {n}(r+\varepsilon)}$, similarly to (\ref{5ytghy78i}), there exists $c_{\varepsilon}$ depending only on $\varepsilon$ such that
\begin{align}\label{465gth51}
&\mathbb{P}_{\delta_{u}}(Z_{\widetilde{{sn}}}(B(\sqrt {n}x,\sqrt {n}r))>0)\cr
&\leq\mathbb{P}_{\delta_{0}}(\bar{M}_{\widetilde{{sn}}}\geq |\sqrt {n}x-u|-\sqrt {n}r)\cr
&\leq\mathbb{P}_{\delta_{0}}(\bar{M}_{\widetilde{{sn}}}\geq \varepsilon\sqrt {n})\cr
&\leq \sum^2_{i=1}\mathbf{P}\left(M''_{\widetilde{{sn}},i}>\frac{\varepsilon\sqrt {n}}{\sqrt 2}\right)+\sum^2_{i=1}\mathbf{P}\left(M'''_{\widetilde{{sn}},i}<-\frac{\varepsilon\sqrt {n}}{\sqrt 2}\right)\cr
&\leq \frac{2c_{\varepsilon}s}{n}.
\end{align}
By (\ref{57hu89k}), if ${|u-\sqrt {n}x|\leq \sqrt {n}(r+\varepsilon)}$, then
\begin{align}\label{656hyhy5}
\mathbb{P}_{\delta_{u}}(Z_{\widetilde{{sn}}}(B(\sqrt {n}x,\sqrt {n}r))>0)\leq \mathbf{P}\left(|Z_{\widetilde{{sn}}}|>0\right)\leq \frac{c}{sn}.
\end{align}
Plugging (\ref{465gth51}) and (\ref{656hyhy5}) into (\ref{5657hy89w}) yields that
\begin{align}\label{465ykopfe4}
&\int_{\mathbb{R}^2} \mathbb{P}_{\delta_{u}}(Z_{\widetilde{{sn}}}(B(\sqrt {n}x,\sqrt {n}r))>0)Z_{\lfloor(1-s)n\rfloor}(du)\cr
&\leq \frac{2c_{\varepsilon}s}{n}Z_{\lfloor(1-s)n\rfloor}(B^c(\sqrt {n}x,\sqrt {n}(r+\varepsilon)))+\frac{c}{sn}Z_{\lfloor(1-s)n\rfloor}(B(\sqrt {n}x,\sqrt {n}(r+\varepsilon))).
\end{align}
Plugging (\ref{465ykopfe4}) into (\ref{454t5t89g}) yields that
\begin{align}\label{47hu901}
&\mathbf{P}(Z_{n}(\sqrt nB(x,r))>0)\cr
&\leq1-\mathbf{E}\left[\exp\left\{-\frac{4c_{\varepsilon}s}{n}Z_{\lfloor(1-s)n\rfloor}(B^c(\sqrt {n}x,\sqrt {n}(r+\varepsilon)))-\right.\right.\cr
&~~~~\left.\left.\frac{2c}{sn}Z_{\lfloor(1-s)n\rfloor}(B(\sqrt {n}x,\sqrt {n}(r+\varepsilon)))\right\}\right]\cr
&= 1-\Big(\mathbb{E}_{n\delta_0}\left[\exp\left\{-\frac{4c_{\varepsilon}s}{n}Z_{\lfloor(1-s)n\rfloor}(B^c(\sqrt {n}x,\sqrt {n}(r+\varepsilon)))-\right.\right.\cr
&~~~~\left.\left.\frac{2c}{sn}Z_{\lfloor(1-s)n\rfloor}(B(\sqrt {n}x,\sqrt {n}(r+\varepsilon)))\right\}\right]\Big)^{1/n}.
\end{align}
It is well-known that under the probability $\mathbb{P}_{n\delta_0}$ (see \cite[p. 75]{lalley2015} for instance), we have
\begin{align}
\lim_{n\to\infty}\left\{\frac{Z_{\lfloor nt\rfloor}(\sqrt n\cdot)}{n}\right\}_{t\geq0}\overset{\text{weakly}}=\{X_t(\cdot)\}_{t\geq0},
\end{align}
with $X_0=\delta_0$. It is easy to check that for a non-negative sequence $\{a_n\}_{n\geq0}$, if $\lim_{n\to\infty}a_n=a\in(0,\infty)$, then $\lim_{n\to\infty}n(1-a_n^{1/n})=-\log a.$ This, combined with (\ref{47hu901}), yields that for any $s\in(0,1)$ and $\varepsilon>0$,
\begin{align}\label{45ghyhu68n}
&\limsup_{n\to\infty} n\mathbb{P}_{\delta_0}(Z_{n}(\sqrt nB(x,r))>0)\cr
&\leq-\log\mathbb{E}_{\delta_0}\Bigg[\exp\left\{-{4c_{\varepsilon}s}X_{1-s}\Big(B^c\big({x},{r+\varepsilon}\big)\Big)-
\frac{2c}{s}X_{1-s}\Big(B\big({x},{r+\varepsilon}\big)\Big)\right\}\Bigg].
\end{align}
\par

We proceed to deal with the r.h.s. (\ref{45ghyhu68n}). Recall that $\lambda$ is the Lebeguese measure on $\mathbb{R}^2$. By (\ref{ujytas}), we have
\begin{align}
&\exp\left\{\int_{\mathbb{R}^2}\log\mathbb{E}_{\delta_0}\Bigg[\exp\left\{-{4c_{\varepsilon}s}X_{1-s}\Big(B^c\big({x},{r+\varepsilon}\big)\Big)-
\frac{2c}{s}X_{1-s}\Big(B\big({x},{r+\varepsilon}\big)\Big)\right\}\Bigg]dx\right\}\cr
&=\mathbb{E}_{\lambda}\left[\exp\Bigg\{X_{1-s}\left(-4c_{\varepsilon}s\ind_{B^c({r+\varepsilon})}-\frac{2c}{s}\ind_{B(r+\varepsilon)}\right)\Bigg\}\right].
\end{align}
From \cite[p. 52]{Etheridge}, we have the following scale invariance property: if $X_0=\lambda$, then for any $\eta,~t>0$ and bounded measurable function $\phi$ on $\mathbb{R}^2$,
\begin{align}\label{pplow2s}
X_t(\phi)\overset{\text{Law}}=\frac{1}{\eta^2}X_{\eta^2 t}(\phi(\cdot/\eta)).
\end{align}
Let $t=1-s$ and $\eta=\frac{1}{\sqrt{1-s}}$, then the scaling property yields that
\begin{align}\label{5yjujuht5q}
&\mathbb{E}_{\lambda}\left[\exp\Bigg\{X_{1-s}\left(-4c_{\varepsilon}s\ind_{B^c({r+\varepsilon})}-\frac{2c}{s}\ind_{B(r+\varepsilon)}\right)\Bigg\}\right]\cr
&=\mathbb{E}_{\lambda}\left[\exp\Bigg\{-\frac{4c_{\varepsilon}s}{1-s}X_{1}\left(\frac{B^c(r+\varepsilon)}{\sqrt{1-s}}\right)\Bigg\}
\exp\Bigg\{-\frac{2c}{s(1-s)}X_1\left(\frac{B(r+\varepsilon)}{\sqrt{1-s}}\right)\Bigg\}\right]\cr
&\geq\mathbb{E}_{\lambda}\left[\exp\Bigg\{-\frac{4c_{\varepsilon}s}{1-s}X_{1}\left(\mathbb{R}^2\right)\Bigg\}
\exp\Bigg\{-\frac{2c}{s(1-s)}X_1\left(B(r+2\varepsilon)\right)\Bigg\}\right],
\end{align}
where the last inequality follows from the fact that since $s\leq \frac{\varepsilon^2}{(r+2\varepsilon)^2}$, we have
$$
\frac{B(r+\varepsilon)}{\sqrt{1-s}}\subset B(r+2\varepsilon).
$$
By the dominated convergence theorem, (\ref{5yjujuht5q}) yields that
\begin{align}\label{54hhyu7uh2}
&\lim_{s\to0+}\mathbb{E}_{\lambda}\left[\exp\Bigg\{X_{1-s}\left(-4c_{\varepsilon}s\ind_{B^c({r+\varepsilon})}-\frac{2c}{s}\ind_{B(r+\varepsilon)}\right)\Bigg\}\right]\cr
&\geq\mathbb{P}_{\lambda}(X_1\left(B(r+2\varepsilon)\right)=0)\cr
&=\mathbb{P}_{\lambda}\left(X_{(r+2\varepsilon)^{-2}}\left(B(1)\right)=0\right)\cr
&=e^{\int_{\mathbb{R}^2}\log \mathbb{P}_{\delta_y}\left(X_{(r+2\varepsilon)^{-2}}\left(B(1)\right)=0\right)dy},
\end{align}
where the first equality comes from (\ref{pplow2s}), and the second equality follows from similar arguments as those of (\ref{4it46fnr4}). Since
$$\int_{\mathbb{R}^2}\log \mathbb{P}_{\delta_y}\left(X_{(r+2\varepsilon)^{-2}}\left(B(1)\right)=0\right)dy=\log \mathbb{P}_{\lambda}(X_1\left(B(r+2\varepsilon)\right)=0),$$
it follows that
$$\int_{\mathbb{R}^2}\log \mathbb{P}_{\delta_y}\left(X_{(r+2\varepsilon)^{-2}}\left(B(1)\right)=0\right)dy$$
is increasing as $\varepsilon\to0+$. Thus, by L\'evy's monotonic convergence theorem and Lemma \ref{4yu7ui89w24}, we have
\begin{align}\label{5hyhyu7r4}
\lim_{\varepsilon\to0+}\int_{\mathbb{R}^2}\log \mathbb{P}_{\delta_y}\left(X_{(r+2\varepsilon)^{-2}}\left(B(1)\right)=0\right)dy&=\int_{\mathbb{R}^2}\lim_{\varepsilon\to0+}\log \mathbb{P}_{\delta_y}\left(X_{(r+2\varepsilon)^{-2}}\left(B(1)\right)=0\right)dy\cr
&=\int_{\mathbb{R}^2}\log \mathbb{P}_{\delta_y}\left(X_{r^{-2}}\left(B(1)\right)=0\right)dy\cr
&=\log\mathbb{P}_{\lambda}(X_1\left(B(r)\right)=0),
\end{align}
where the last equality follows from the last equality of (\ref{54hhyu7uh2}).

\par

Finally, by (\ref{5hythtt7j4r}), we have for any $\varepsilon>0,~s\in\left(0,\frac{\varepsilon^2}{(r+2\varepsilon)^2}\right)$,
\begin{align}\label{65hgyjuye3tb}
&\liminf_{n\to\infty}\mathbb{P}\left(\frac{R_n}{\sqrt n }\geq r\right)\cr
&\geq \exp\left\{-\int_{\mathbb{R}^2}\limsup_{n\to\infty}n\mathbb{P}_{\delta_0}(Z_n(\sqrt nB(x,r))>0)dx\right\}\cr
&\geq\exp\left\{\int_{\mathbb{R}^2}\log\mathbb{E}_{\delta_0}\Bigg[\exp\left\{-{4c_{\varepsilon}s}X_{1-s}\Big(B^c\big({x},{r+\varepsilon}\big)\Big)-
\frac{2c}{s}X_{1-s}\Big(B\big({x},{r+\varepsilon}\big)\Big)\right\}\Bigg]dx\right\}\cr
&=\mathbb{E}_{\lambda}\left[\exp\Bigg\{X_{1-s}\left(-4c_{\varepsilon}s\ind_{B^c({r+\varepsilon})}-\frac{2c}{s}\ind_{B(r+\varepsilon)}\right)\Bigg\}\right],
\end{align}
where the second inequality comes from (\ref{45ghyhu68n}).
Taking $s\to0+$ on the both sides of (\ref{65hgyjuye3tb}) and then using (\ref{54hhyu7uh2}) yields that for any $\varepsilon>0$,
\begin{align}\label{45gtrhyt6u}
\liminf_{n\to\infty}\mathbb{P}\left(\frac{R_n}{\sqrt n }\geq r\right)\geq e^{\int_{\mathbb{R}^2}\log \mathbb{P}_{\delta_y}\left(X_{(r+2\varepsilon)^{-2}}\left(B(1)\right)=0\right)dy}.
\end{align}
Taking $\varepsilon\to0+$ on both sides of (\ref{45gtrhyt6u}) and then using (\ref{5hyhyu7r4}) yields that
\begin{align}\label{45gtrhytfbgngdg6u}
\liminf_{n\to\infty}\mathbb{P}\left(\frac{R_n}{\sqrt n }\geq r\right)
&\geq\mathbb{P}_{\lambda}(X_1\left(B(r)\right)=0)\cr
&=\exp\left\{\int_{\mathbb{R}^2}\log\mathbb{P}_{\delta_0}\left(X_1(B(x,r))=0\right)dx\right\}.
\end{align}
This, combined with (\ref{5ty7hy61h}), concludes Theorem \ref{thdim2}.
\end{proof}
\section{Proof of Theorem \ref{thdim4}: Empty Ball With Infinite Variance}\label{8i87ujr4g}
This section aims to study the empty ball of the branching random walk with infinite variance offspring law. The proofs are divided into three cases according to $d\beta<2$, $d\beta=2$ and $d\beta>2$.
\bigskip

\noindent{\bf Case~1:~$d\beta<2.$} The proof is a modification of \cite[Theorem 1.8]{xzacta}. So, we just put some additional steps here.
\begin{proof}Recall that $b_n=[nL(\mathbb{P}_{\delta_0}(Z_n(\mathbb{R}^d)>0))]^{\frac{1}{\beta d}},~n\geq1$, where $L(\cdot)$ is a slowly varying function at $0$. Similarly to (\ref{897ojki2fn}), we have
\begin{align}\label{5hf6787i8yt56}
\mathbb{P}(R_n\geq b_nr)=\exp\left\{-\int_{\mathbb{R}^d}\mathbb{P}_{\delta_x}(Z_n(B(b_nr))>0)dx\right\}.
\end{align}
 We first claim that
\begin{align}\label{4rt5tgtl}
\mathbb{P}_{\delta_x}\left(Z_n(B(b_nr))>0\Big{|}|Z_n|>0\right)\geq \mathbb{P}(|x+S_n|\leq b_nr).
\end{align}
In fact, under the event $\{|Z_n|>0\}$, there exist many random walks (at least one) alive at time $n$. So, the probability of exists one particle located in $B(b_nr)$ at time $n$ is larger than the probability of exactly one random walk located in $B(b_nr)$.
For the upper bound of $\mathbb{P}(R_n/b_n\geq r)$, it is simple to see that
\begin{align}\label{54hyt7uhy}
&\int_{\mathbb{R}^d}\mathbb{P}_{\delta_x}(Z_n(B(b_nr))>0)dx\cr
&=\int_{\mathbb{R}^d}\mathbb{P}_{\delta_x}(|Z_n|>0)\mathbb{P}_{\delta_x}\left(Z_n(B(b_nr))>0\Big{|}|Z_n|>0\right)dx\cr
&\geq \int_{\mathbb{R}^d}\mathbb{P}_{\delta_x}(|Z_n|>0)\mathbb{P}(|x+S_n|\leq b_nr)dx\cr
&=\mathbb{P}_{\delta_0}(|Z_n|>0)\mathbb{E}\left[\int_{\mathbb{R}^d}\ind_{\{|x+S_n|\leq b_nr\}}dx\right]\cr
&=\mathbb{P}_{\delta_0}(|Z_n|>0)b_n^dv_d(r),
\end{align}
where the first inequality follows from (\ref{4rt5tgtl}). By \cite[Lemma 2]{slack1968}, we have
\begin{align}\label{p0kq2gz}
\lim_{n\to\infty}\mathbb{P}_{\delta_0}(|Z_n|>0)b_n^d=(\beta^{-1})^{\beta^{-1}}.
\end{align}
This, combined with (\ref{54hyt7uhy}), yields that
\begin{align}\label{90o0e3aq2}
\lim_{n\to\infty}\mathbb{P}\left(\frac{R_n}{b_n}\geq r\right)\leq \exp\left\{-v_d(r)(\beta^{-1})^{\beta^{-1}}\right\}.
\end{align}

\par
We proceed to deal with the lower bound. Fix $\delta>0$. By the Markov inequality, we have
\begin{align}\label{54gfhyu7u7de}
\int_{\mathbb{R}^d}\mathbb{P}_{\delta_x}(Z_n(B(b_nr))>0)dx
&\leq \int_{|x|<b_n(r+\delta)}\mathbb{P}_{\delta_0}(|Z_n|>0)dx+\int_{|x|\geq b_n(r+\delta)}\mathbb{E}_{\delta_x}[Z_n(B(b_nr))]dx\cr
&\leq \mathbb{P}_{\delta_0}(|Z_n|>0)b_n^dv_d(r+\delta)+\int_{|x|\geq b_n(r+\delta)}\mathbb{P}_{x}(S_n\in B(b_nr))dx.
\end{align}
\par
Next, we are going to show that the second term on the r.h.s. of (\ref{54gfhyu7u7de}) tends to $0$ as $n\to\infty$. Copied from \cite[Corollary 1.8]{Nagaev}, if $\mathbb{E}[|X|^{\alpha}]<\infty$ for some $\alpha\geq2$, then there exist positive constants $c_{\alpha}$ and $c'_{\alpha}$ such that for any $n\geq 1$ and $x>0$,
\begin{align}\label{465juiqw2}
\mathbb{P}(|S^{(1)}_n|\geq x)\leq \frac{c_{\alpha}n}{x^{\alpha}}+2\exp\{-c'_{\alpha}x^2/n\}.
\end{align}
By Fubini's theorem,
\begin{align}\label{45gthypgd4}
&\int_{|x|\geq b_n(r+\delta)}\mathbb{P}_{x}(S_n\in(B(b_nr)))dx\cr
&=\mathbb{E}\left[\int_{\mathbb{R}^d}\ind_{B^c(b_n(r+\delta))}(x)\ind_{\{|x+S_n|\leq b_nr\}}dx, |S_n|\geq b_n\delta\right]\cr
&\leq v_d(b_nr)\mathbb{P}\left(|S_n|\geq b_n\delta\right)\cr
&\leq db_n^dv_d(r)\mathbb{P}\left(|S^{(1)}_n|\geq b_n\delta/\sqrt{d}\right)\cr
&\leq dd^{\alpha/2}v_d(r)b_n^d\frac{c_{\alpha}n}{(b_n\delta)^{\alpha}}+2dv_d(r)b_n^d\exp\{-c'_{\alpha}b_n^2\delta^2/(dn)\},
\end{align}
where the last inequality follows from (\ref{465juiqw2}). If $\alpha>(\beta+1)d$, then the first term on the r.h.s. of (\ref{45gthypgd4}) tends to $0$. If $\beta d<2$, then the second term on the r.h.s. of (\ref{45gthypgd4}) tends to $0$. Thus, under the assumption of Theorem \ref{thdim4} (i), we have
\begin{align}\label{4gfhyu7hfr4e3}
\lim_{n\to\infty}\int_{|x|\geq b_n(r+\delta)}\mathbb{P}_{x}(S_n\in B(b_nr))dx=0.
\end{align}
\par
Finally, plugging (\ref{4gfhyu7hfr4e3}) into (\ref{54gfhyu7u7de}) and then using (\ref{p0kq2gz}) entails that for any $\delta>0$,
\begin{align}\label{54hyu7i8gt5}
\lim_{n\to\infty}\int_{\mathbb{R}^d}\mathbb{P}_{\delta_x}(Z_n(B(b_nr))>0)dx\leq v_d(r+\delta)(\beta^{-1})^{\beta^{-1}}.
\end{align}
Letting $\delta\to0+$ in (\ref{54hyu7i8gt5}), then pugging it into (\ref{5hf6787i8yt56}) yields that
\begin{align}
\lim_{n\to\infty}\mathbb{P}\left(\frac{R_n}{b_n}\geq r\right)\geq \exp\left\{-v_d(r)(\beta^{-1})^{\beta^{-1}}\right\}.
\end{align}
This, together with (\ref{90o0e3aq2}), concludes Theorem \ref{thdim4}.
\end{proof}
\bigskip

\noindent {\bf Case~2:~$d\beta=2.$} The proof of this case is similar to the proof of the $2$-dimensional branching random walk with finite variance offspring law (i.e. Theorem \ref{thdim2}).
\begin{proof}
 If $d\beta=2$, then
\begin{align}
\mathbb{P}(R_n\geq \sqrt n r)
&=\exp\left\{-\int_{\mathbb{R}^d}\mathbb{P}_{\delta_x}(Z_n(B(\sqrt n r))>0)dx\right\}\cr
&=\exp\left\{-\int_{\mathbb{R}^d}n^{d/2}\mathbb{P}_{\delta_0}(Z_n(\sqrt nB(y,r))>0)dy\right\}\cr
&=\exp\left\{-\int_{\mathbb{R}^d}n^{\frac{1}{\beta}}\mathbb{P}_{\delta_0}(Z_n(\sqrt nB(x,r))>0)dx\right\}.
\end{align}
Assume $\lim_{n\to\infty} n^{^{1+\beta}}\sum^{\infty}_{k=n}p_k=\kappa(\beta)$ with $\beta\in(0,1)$. It is well-known that (e.g. \cite[(1.15)]{hou2024}), under the probability $\mathbb{P}_{\lfloor n^{\frac{1}{\beta}}\rfloor\delta_0}$, we have
\begin{align}
\lim_{n\to\infty}\left\{\frac{Z_{\lfloor nt\rfloor}(\sqrt n\cdot)}{n^{\frac{1}{\beta}}}\right\}_{t\geq0}\overset{\text{weakly}}=\{X_t(\cdot)\}_{t\geq0},
\end{align}
where $\{X_t(\cdot)\}_{t\geq0}$ is a super-Brownian motion with initial value $X_0=\
\delta_0$ and stable branching mechanism $\psi(u)=\frac{1}{\beta}{\kappa(\beta)\Gamma(1-\beta)}u^{1+\beta}$.
The following steps are almost the same as the proof of Theorem \ref{thdim2}. The main changes are to replace $n$ with $n^{\frac{1}{\beta}}$ and replace $\mathbb{R}^2$ with $\mathbb{R}^d$. So, we feel free to omit the proofs here.
\end{proof}

\medskip

\noindent {\bf Case~3:~$d\beta>2.$} One can observe that the result of this case is parallel with \cite[Theorem 1.5]{xzacta}. However, since the offspring law may have infinite variance, the Paley-Zygmund inequality used in \cite[Theorem 1.5]{xzacta} does not work anymore. To overcome this difficulty, we use the spine decomposition method developed in Rapenne \cite[Lemma 6.2]{rapenne2023}. Nonetheless, since our underlying motion is a general random walk rather than the stable process studied in \cite[Lemma 6.2]{rapenne2023}, our proofs are more delicate.
\par Let's give a brief introduction to the spine decomposition method. Let $\{\mathcal{F}_n\}_{n\geq0}$ be the natural filtration of $\{|Z_n|\}_{n\geq0}$. Since $(\{|Z_n|\}_{n\geq0},\{\mathcal{F}_n\}_{n\geq0},\mathbf{P})$ is a martingale with mean $1$, we can define the probability measure $\mathbf Q$ through
\begin{align}
\frac{d\mathbf Q}{d\mathbf{P}}\Big{|}\mathcal{F}_n:=|Z_n|.
\end{align}
Let $p^*_k:=kp_k,~k\geq0$. Define the size biased branching random walk as follows. We start with one particle called $w_0$ located at the origin in $\mathbb{R}^d$. It defines the generation $0$ of the size biased branching random walk. For generation $n\geq1$, we define it recursively. Let $Z_n$ be the set of particles in generation $n$. We also have a marked
particle $w_n$ among $Z_n$. The particle $w_n$ produce children according to the offspring law $\{p^*_k\}_{k\geq0}$ (note that since $p^*_0=0$,  $w_n$ has at least one children). The children of $w_n$ jump independently from the position of $w_n$ according to the law of the step size $X$ under $\mathbf{P}$. Among the children of $w_n$, we choose uniformly at random a special particle called $w_{n+1}$. Moreover, every particle $u\neq w_n$ at generation $n$ gives birth independently according to the law of $\{p_k\}_{k\geq0}$. The offsprings of $u$
jump from the position of $u$ according to the law of $X$ under $\mathbf{P}$ independently. The offsprings of $Z_n$ with their new positions (including the marked particle $w_{n+1}$) form the $(n + 1)$-th generation $Z_{n+1}$. Let $\{\mathcal{F}^*_n\}_{n\geq0}$ be the natural filtration of $(\{Z_n\}_{n\geq0}, \{w_n\}_{n\geq0})$ and $Q^*$ be the extension of the measure $Q$ on $\{\mathcal{F}^*_n\}_{n\geq0}$. With this construction, we call $(\{Z_n\}_{n\geq0}, \{w_n\}_{n\geq0},\{\mathcal{F}^*_n\}_{n\geq0}, Q^*)$ is a size biased branching random walk with spine $\{w_n\}_{n\geq0}$. From this definition, we have
for every $n\geq1$, $Q^*|\mathcal{F}_n=Q|\mathcal{F}_n$. Moreover, for every particle $u$ at generation $n$,
\begin{align}\label{5hyy7fre31}
Q^*( w_n=u|\mathcal{F}_n)=\frac{1}{|Z_n|}.
\end{align}
Recall that $\{Z^u_n\}_{n\geq0}$ is the sub-branching random walk emanating from particle $u$ with $Z^u_0=\delta_0$. $X_u$ is the displacement of particle $u$. $b(w_k)$ is the set of brothers of $w_k$. $S_{w_k}$ is the position of $w_k$. For $(\{Z_n\}_{n\geq0}, \{w_n\}_{n\geq0},\{\mathcal{F}^*_n\}_{n\geq0}, Q^*)$, it is simple to see that for any $n\geq1$, measurable set $C$ and $x\in\mathbb{R}^d$,
\begin{align}\label{3rftgt8j2g}
Z_n(C-x)=\ind_{\{S_{w_n\in C-x}\}}+\sum^{n-1}_{k=0}\sum_{u\in b(w_{k+1})}Z^u_{n-k-1}(C-x-S_{w_k}-X_u),
\end{align}
which is called the {\em spine decomposition.} For more detailed description; see \cite[Section 2.4]{rapenne2023}.
\par
Now, we are ready to prove Theorem \ref{thdim4} (iii).
\begin{proof}
Similarly to (\ref{897ojki2fn}), we have
\begin{align}\label{34opmkq21}
\mathbb{P}(R_n\geq r)=\exp\left\{-\int_{\mathbb{R}^d}\mathbb{P}_{\delta_x}(Z_n(B(r))>0)dx\right\}.
\end{align}
Fix $r>0$. Let $A:=B(r)$ for short. Thus, to prove Theorem \ref{thdim4}, it suffices to calculate $\lim_{n\to\infty}\int_{\mathbb{R}^d}\mathbf{P}(Z_n(A-x)>0)dx$.
\par
We first express $\mathbf{P}(Z_n(A-x)>0)$ as a functional of a size biased branching random walk by the spine decomposition method. Fix $n\geq1$.
Let
\begin{align}\label{5gjidwa2rf}
Y^n_k(z)=\sum_{u\in b(w_{k+1})}Z^u_{n-k-1}(A-z-X_u), 0\leq k\leq n-1.
\end{align}
Note that $\{b(w_{k+1})\}_{k\geq0}$ are i.i.d.  and independent of $\{Z^u_{k-1}(A-z-X_u)\}_{\{u\in b(w_k), k\geq0\}}$. Moreover,  $\{Z^u_{k-1}(A-z-X_u)\}_{\{u\in b(w_k), k\geq0\}}$ are i.i.d.. Thus, for $n+1\geq k\geq1$ and $z\in\mathbb{R}^d$,
\begin{align}\label{5yhyuhyr41}
Y^n_{n-k}(z)=\sum_{u\in b(w_{n-k+1})}Z^u_{k-1}(A-z-X_u)\overset{\text{Law}}=\sum_{u\in b(w_{k+1})}Z^u_{k-1}(A-z-X_u)=:Y_k(z).
\end{align}By (\ref{5hyy7fre31}), we have $Q^*$-a.s. (note that $Q^*(|Z_n|\geq1)=Q(|Z_n|\geq1)=1$),
\begin{align}
{Q^*}(S_{w_n}\in A-x|\mathcal{F}_n)=\frac{Z_n(A-x)}{|Z_n|}.
\end{align}
Define $\frac{0}{0}=:0$. Thus,
\begin{align}
\mathbb{E}_{Q^*}\left[\frac{\ind_{\{S_{w_n}\in A-x\}}}{Z_n(A-x)}\right]
&=\mathbb{E}_{Q^*}\left[\frac{\ind_{\{Z_n(A-x)\geq 1\}}}{Z_n(A-x)}\mathbb{E}_{Q^*}\left[{\ind_{\{S_{w_n}\in A-x\}}}\Big{|}\mathcal{F}_n\right]\right]\cr
&=\mathbb{E}_{Q^*}\left[\frac{\ind_{\{Z_n(A-x)\geq 1\}}}{Z_n(A-x)}\frac{Z_n(A-x)}{|Z_n|}\right]\cr
&=\mathbb{E}_{Q}\left[\frac{\ind_{\{Z_n(A-x)\geq 1\}}}{|Z_n|}\right]\cr
&=\mathbf{P}(Z_n(A-x)\geq 1).
\end{align}
This, combined with (\ref{3rftgt8j2g}), yields that
\begin{align}\label{4tggt68gtak3}
\mathbf{P}(Z_n(A-x)\geq 1)&=\mathbb{E}_{Q^*}\left[\frac{\ind_{\{S_{w_n}\in A-x\}}}{Z_n(A-x)}\right]\cr
&=\mathbb{E}_{Q^*}\left[\frac{\ind_{\{S_{w_n}\in A-x\}}}{\ind_{\{S_{w_n\in A-x}\}}+\sum^{n-1}_{k=0}\sum_{u\in b(w_{k+1})}Z^u_{n-k-1}(A-x-S_{w_k}-X_u)}\right]\cr
&=\mathbb{E}_{Q^*}\left[\frac{\ind_{\{S_{w_n}\in A-x\}}}{1+\sum^{n-1}_{k=0}Y^n_k(x+S_{w_k})}\right]\cr
&=\mathbb{E}_{Q^*}\left[\frac{\ind_{\{S_{w_n}\in A-x\}}}{1+\sum^{n}_{i=1}Y^n_{n-i}(x+S_{w_{n-i}})}\right]\cr
&=\mathbb{E}_{Q^*}\left[\frac{\ind_{\{S_{w_n}\in A-x\}}}{1+\sum^{n}_{k=1}Y_{k}(x+S_{w_{n-k}})}\right]\cr
&=\mathbb{E}_{Q^*}\left[\frac{\ind_{\{S_{w_n}\in A-x\}}}{1+\sum^{n}_{k=1}Y_{k}(x+S_{w_{n}}-S_{w_{k}})}\right].
\end{align}
The third equality of (\ref{4tggt68gtak3}) comes from (\ref{5gjidwa2rf}). The penultimate equality follows from (\ref{5yhyuhyr41}) and the fact that $\{S_{w_{i}}\}_{i\geq1}$ and $\{Y_{i}(z)\}_{i\geq1}$ are independent. The last equality follows from the fact that the backward random walk $\{S_{w_n}-S_{w_{n-i}}\}_{0\leq i\leq n}$ has the same distribution as $\{S_{w_i}\}_{0\leq i\leq n}$ and independent of $\{Y_k(z)\}_{k\geq0}.$
\par

Next, we split $\int_{\mathbb{R}^d}\mathbf{P}(Z_n(A-x)>0)dx$ into three parts. By (\ref{4tggt68gtak3}), we have for any $M>0$ and integers $n\geq K\geq1$,
\begin{align}\label{46yhj89012dr4}
&\int_{\mathbb{R}^d}\mathbf{P}(Z_n(A-x)>0)dx\cr
&=\int_{|x|>M\sqrt n}\mathbf{P}(Z_n(A-x)>0)dx+\int_{|x|\leq M\sqrt n}  \mathbb{E}_{Q^*}\left[\frac{\ind_{\{S_{w_n}\in A-x\}}}{1+\sum^{n}_{k=1}Y_{k}(x+S_{w_{n}}-S_{w_k})}\right]dx\cr
&=\int_{|x|>M\sqrt n}\mathbf{P}(Z_n(A-x)>0)dx+\int_{|x|\leq M\sqrt n}I_{K,n}(x)dx+\cr
&~~~~\int_{|x|\leq M\sqrt n}\mathbb{E}_{Q^*}\left[\frac{\ind_{\{S_{w_n}\in A-x\}}}{1+\sum^{K}_{k=1}{Y}_k(x+S_{w_n}-S_{w_k})})\right]dx,
\end{align}
where
\begin{align}\label{4gtrtrww2v}
I_{K,n}(x):=\mathbb{E}_{Q^*}\left[\frac{\ind_{\{S_{w_n}\in A-x\}}}{1+\sum^{n}_{k=1}Y_{k}(x+S_{w_{n}}-S_{w_k})}-\frac{\ind_{\{S_{w_n}\in A-x\}}}{1+\sum^{K}_{k=1}Y_{k}(x+S_{w_{n}}-S_{w_k})}\right].
\end{align}
Hence, to prove the theorem, it suffices to study asymptotics of the three terms on the r.h.s. of (\ref{46yhj89012dr4}) respectively.
\par
We first deal with the first term on the r.h.s. of (\ref{46yhj89012dr4}). Recall that $A=B(r)$. By the Markov inequality and Fubini's theorem, we have
\begin{align}\label{565yh9op1f}
\int_{|x|> M\sqrt n}\mathbf{P}(Z_n(A-x)>0)dx
&\leq \int_{|x|> M\sqrt n}\mathbb{E}_{\delta_x}[Z_n(B(r))]dx\cr
&= \int_{|x|> M\sqrt n}\mathbb{P}(|S_n+x|< r)dx\cr
&= \int_{|x|> M\sqrt n} \int_{\mathbb{R}^d}\ind_{\{|y+x|<r\}}\mathbb{P}(S_n\in dy)dx\cr
&\leq\int_{|y|\geq M\sqrt n-r}\int_{\mathbb{R}^d}\ind_{\{|y+x|<r\}}dx\mathbb{P}(S_n\in dy)\cr
&=v_d(r)\mathbb{P}(|S_n|\geq M\sqrt n-r).
\end{align}
By the central limit theorem, above yields that
\begin{align}\label{5tey7uyu672f}
\lim_{M\to\infty}\lim_{n\to\infty}\int_{|x|> M\sqrt n}\mathbf{P}(Z_n(A-x)>0)dx=0.
\end{align}
\par

We proceed to deal with the second term on the r.h.s. of (\ref{46yhj89012dr4}).  Denote by
$\mathcal{G}$ the $\sigma$-field generated by the spine and number of children of particles in the spine
and positions of the brothers of the spine. Since the covariance matrix of $X$ is invertible, by \cite[p. 218, Corollary 1]{stone1967}, it entails that there exists a constant $c>0$ depending only on $A$ such that for any $x\in\mathbb{R}^d$, $n\geq k\geq1$ and particle $u$,
\begin{align}\label{46u7pl1d}
&\mathbb{E}_{Q^*}[Z^u_{k-1}(A-x-S_{w_{n-k}}-X_u)|\mathcal{G}]\cr
&=\mathbb{P}(S_{k-1}\in A-z-\rho)\big{|}_{z=x+S_{w_{n-k}},\rho=X_u}\cr
&\leq \frac{c}{{k}^{d/2}}.
\end{align}
Since $(\sum^n_{k=K+1}\lambda_i)^{\beta}\leq\sum^n_{k=K+1}\lambda^{\beta}_i$ for any non-negative $\lambda_i$ and $\beta\in(0,1)$, by the Markov inequality, there exists a constant $C_A$ depending only on $A$ such that for any $n\geq K+1\geq1$ and $x\in\mathbb{R}^d$,
\begin{align}\label{45hyuj8ow2}
&Q^*\left(S_{w_n}\in A-x,\sum^n_{k=K+1}Y_k(x+S_{w_{n-k}})\geq 1\right)\cr
&\leq\mathbb{E}_{Q^*}\left[\ind_{\{S_{w_n}\in A-x\}}\sum^n_{k=K+1}\mathbb{E}_{Q^*}\left[Y_k(x+S_{w_{n-k}})^{\beta}|\mathcal{G}\right]\right]\cr
&\leq\sum^n_{k=K+1}\mathbb{E}_{Q^*}\left[\ind_{\{S_{w_n}\in A-x\}}\left(\sum_{u\in b(w_{k+1})}\mathbb{E}_{Q^*}[Z^u_{k-1}(A-x-S_{w_{n-k}}-X_u)|\mathcal{G}]\right)^{\beta}\right]\cr
&\leq\sum^n_{k=K+1}\frac{c^{\beta}}{k^{d\beta/2}}\mathbb{E}_{Q^*}\left[\ind_{\{S_{w_n}\in A-x\}} |b(w_{k+1})|^{\beta}\right]\cr
&\leq C_An^{-d/2}\sum^{+\infty}_{k=K+1}\frac{1}{k^{d\beta/2}}.
\end{align}
In (\ref{45hyuj8ow2}), the second inequality follows from Jensen's inequality. The third inequality follows from (\ref{46u7pl1d}), where $|b(w_{k+1})|$ stands for the cardinality of the set $b(w_{k+1})$. The last equality follows from the fact that
$$\mathbb{E}_{Q^*}\left[\ind_{\{S_{w_n}\in A-x\}} |b(w_{k+1})|^{\beta}\right]\leq \mathbb{E}_{Q^*}\left[|b(w_{k+1})|^{\beta}\right]=\sum^{\infty}_{k=1}(k-1)^{\beta}p^*_{k}\leq \sum^{\infty}_{k=1}k^{1+\beta}p_k<\infty.$$
Since $\beta d>2$, the summation on the r.h.s. of (\ref{45hyuj8ow2}) is finite. Thus, by (\ref{4gtrtrww2v}), it follows that
\begin{align}\label{465gty912}
&\lim_{K\to\infty}\lim_{n\to\infty}\Big{|}\int_{|x|\leq M\sqrt n}I_{K,n}(x)dx\Big{|}\cr
&\leq \lim_{K\to\infty}\lim_{n\to\infty}\int_{|x|\leq M\sqrt n}Q^*\left(S_{w_n}\in A-x,\sum^n_{k=K+1}Y_k(x+S_{w_{n-k}})\geq 1\right)dx\cr
&\leq\lim_{K\to\infty}\lim_{n\to\infty} \sum^{+\infty}_{k=K+1}\frac{C_A}{k^{d\beta/2}}n^{-d/2}\int_{|x|\leq M\sqrt n}dx=0.
\end{align}

\par
It remains to deal with the third term on the r.h.s. of (\ref{46yhj89012dr4}). It is easy to see that for $n$ satisfying $M\sqrt n/2>r+|S_{w_K}|$, on the event $\{|S_{w_n}-S_{w_K}|\leq M\sqrt n/2\}$, it follows that
\begin{align}\label{4gftbgbtgt5}
A-S_{w_K}-(S_{w_n}-S_{w_K})\subset B(M\sqrt n).
\end{align}
Recall that $\lambda$ is the $d$-dimensional Lebesgue measure. Recall that $\mathcal{F}^*_k$ is the $\sigma$-algebra which contains the information of the size biased branching random walk $(\{Z_n\}_{n\geq0}, \{w_n\}_{n\geq0})$ up to time $k$. Set $\mathbb{P}_{\mathcal{F}^*_k}(\cdot):=\mathbb{P}(\cdot|{\mathcal{F}^*_k})$, $\mathbb{E}_{\mathcal{F}^*_k}[\cdot]:=\mathbb{E}[\cdot|{\mathcal{F}^*_k}].$ By Fubini's theorem, if $n>(2(r+|S_{w_K}|)/M)^2$, then for any measurable set $B\subset A$, we have
\begin{align}
&\int_{|x|\leq M\sqrt n}\mathbb{P}_{\mathcal{F}^*_K}(S_{w_n}-S_{w_K}+x+S_{w_K}\in B)dx\cr
&=:\int_{|x|\leq M\sqrt n}\mathbb{P}_{\mathcal{F}^*_K}(S_{w_n}-S_{w_K}+x+S_{w_K}\in B,|S_{w_n}-S_{w_K}|\leq M\sqrt n/2)dx+H_n(B)\cr
&=\mathbb{E}_{\mathcal{F}^*_K}\left[\int_{|x|\leq M\sqrt n}\ind_{\{x\in B-S_{w_n}+S_{w_K}-S_{w_K}\}}dx,|S_{w_n}-S_{w_K}|\leq M\sqrt n/2\right]+H_n(B)\cr
&=\mathbb{E}_{\mathcal{F}^*_K}\left[\lambda\Big{(}B(M\sqrt n)\cap \big{(}B-S_{w_K}-(S_{w_n}-S_{w_K})\big{)}\Big{)},|S_{w_n}-S_{w_K}|\leq M\sqrt n/2\right]+H_n(B)\cr
&=\mathbb{E}_{\mathcal{F}^*_K}\left[\lambda(B),|S_{w_n}-S_{w_K}|\leq M\sqrt n/2\right]+H_n(B)\cr
&=\lambda(B)\mathbb{P}(|S_{w_n}-S_{w_K}|\leq M\sqrt n/2)+H_n(B),
\end{align}
where the fourth equality follows from (\ref{4gftbgbtgt5}). Hence, for $n>(2(r+|S_{w_K}|)/M)^2$, the third term on the r.h.s. of (\ref{46yhj89012dr4}) equals to
\begin{align}\label{466y9i21h}
&\int_{|x|\leq M\sqrt n}\mathbb{E}_{Q^*}\left[\frac{\ind_{\{S_{w_n}\in A-x\}}}{1+\sum^{K}_{k=1}{Y}_k(x+S_{w_n}-S_{w_k})})\right]dx\cr
&=\int_{|x|\leq M\sqrt n}\mathbb{E}_{Q^*}\left[\int_{A}\frac{\mathbb{P}_{\mathcal{F}^*_K}(S_{w_n}-S_{w_K}+x+S_{w_K}\in dy)}{1+\sum^{K}_{k=1}{Y}_k(y-S_{w_k})}\right]dx\cr
&=\mathbb{E}_{Q^*}\left[\int_{A}\frac{1}{1+\sum^{K}_{k=1}{Y}_k(y-S_{w_k})}dy\right]\mathbb{P}(|S_{w_n}-S_{w_K}|\leq M\sqrt n/2)+\cr
&~~~~\mathbb{E}_{Q^*}\left[\int_{A}\frac{1}{1+\sum^{K}_{k=1}{Y}_k(y-S_{w_k})}H_n(dy)\right]\cr
&=:Q_1+Q_2.
\end{align}
Since $H_n(B)\leq \lambda(B)\mathbb{P}(|S_{w_n}-S_{w_K}|>M\sqrt n/2)$, we have
\begin{align}\label{45g6734gty6}
\lim_{M\to\infty}\lim_{n\to\infty}Q_2
&\leq \lim_{M\to\infty}\lim_{n\to\infty}\mathbb{E}_{Q^*}\left[H_n(A)\right]\cr
&\leq\lim_{M\to\infty}\lim_{n\to\infty}\lambda(A)\mathbb{P}(|S_{w_n}-S_{w_K}|>M\sqrt n/2)=0,
\end{align}
where the last equality follows from the central limit theorem. For $Q_1$, it is simple to see that
\begin{align}\label{46yhy89j2}
\lim_{M\to\infty}\lim_{K\to\infty}\lim_{n\to\infty}Q_1=\mathbb{E}_{Q^*}\left[\int_{A}\frac{1}{1+\sum^{\infty}_{k=1}{Y}_k(y-S_{w_k})}dy\right]=:I_A,
\end{align}
where the first equality follows from applying the central limit theorem as $n\to\infty$ and then the dominated convergence theorem as $K\to\infty$. Plugging (\ref{46yhy89j2}) and (\ref{45g6734gty6}) into (\ref{466y9i21h}) yields that
\begin{align}\label{45hyhy91f}
\lim_{M\to\infty}\lim_{K\to\infty}\lim_{n\to\infty}\int_{|x|\leq M\sqrt n}\mathbb{E}_{Q^*}\left[\frac{\ind_{\{S_{w_n}\in A-x\}}}{1+\sum^{K}_{k=1}{Y}_k(x+S_{w_n}-S_{w_k})})\right]dx=I_A.
\end{align}
\par
We claim that $I_A\in(0,\infty)$. In fact, by Fubini's theorem, we can easily obtain that $I_A\leq \lambda(A)<\infty.$ We proceed to prove $I_A>0$. It is simple to see that for $y\in A=B(r)$, we have
\begin{align}\label{4rr6yu1}
Y_k(y-S_{w_k})&=\sum_{u\in b(w_{k+1})}Z^u_{k-1}(A-y+S_{w_k}-X_u)\cr
&\leq \sum_{u\in b(w_{k+1})}Z^u_{k-1}(B(2r)+S_{w_k}-X_u)=:\hat{Y}_k(-S_{w_k}).
\end{align}
Using similar arguments to obtain (\ref{45hyuj8ow2}), it yields that there exists a constant $c_r$ depending only on $r$ such that for any $k\geq1$,
\begin{align}
Q^*(\hat{Y}_k(-S_{w_k})\geq1)\leq \frac{c_r}{k^{d\beta/2}},\nonumber
\end{align}
which is summable w.r.t. $k$. Hence, by the Borel-Cantelli's lemma, $Q^*$-almost surely, there is only a finite number of integers $k$ such that $\hat{Y}_k(-S_{w_k})\geq1$. Namely,
\begin{align}\label{5gui9w2}
\sum^{\infty}_{k=1}\hat{Y}_k(-S_{w_k})<\infty,~Q^*\text{-a.s.}
\end{align}
This, together with Fubini's theorem, implies that
\begin{align}
I_A&=\int_{A}\mathbb{E}_{Q^*}\left[\frac{1}{1+\sum^{\infty}_{k=1}{Y}_k(y-S_{w_k})}\right]dy\cr
&=\int^{\infty}_0\frac{1}{(1+t)^2}\int_{A}{Q^*}\left(\sum^{\infty}_{k=1}{Y}_k(y-S_{w_k})\in[0,t]\right)dydt\cr
&\geq \lambda(A)\int^{\infty}_0\frac{1}{(1+t)^2}{Q^*}\left(\sum^{\infty}_{k=1}\hat{Y}_k(-S_{w_k})\in[0,t]\right)dt \cr
&>0,
\end{align}
where the first inequality follows from (\ref{4rr6yu1}) and the last inequality follows from (\ref{5gui9w2}).
\par
Finally, taking first $n\to\infty$, then $K\to\infty$ and at last $M\to\infty$ on the both sides of (\ref{46yhj89012dr4}), and plugging (\ref{5tey7uyu672f}), (\ref{465gty912}) and (\ref{45hyhy91f}) into (\ref{46yhj89012dr4}) yields that
\begin{align}
\lim_{n\to\infty}\int_{\mathbb{R}^d}\mathbf{P}(Z_n(A-x)>0)dx=I_A.\nonumber
\end{align}
This, combined with (\ref{34opmkq21}), concludes that
\begin{align}
\lim_{n\to\infty}\mathbb{P}(R_n\geq r)=e^{-I_A}\in(0,1).\nonumber
\end{align}
We have finished the proof of Theorem \ref{thdim4} (iii).
\end{proof}

\textbf{Acknowledgements} I am very grateful to Prof. Jie Xiong, whose joint works \cite{xz21}\cite{xzacta} with me have laid the foundation for the present. I would also like to extend my special thanks to Dr. Haojie Hou, who has provided many concrete and effective ideas in this subject. My sincere thanks go to the two referees: their careful reading of the manuscript and many suggestions have improved this paper significantly. Shuxiong Zhang is supported by NSF of China (No. 12501176), and by the National Key R\&D Program of China 2022YFA1006102.

\end{document}